\documentclass[11pt]{article}
\usepackage[margin =1in]{geometry}
\usepackage{amssymb,amsthm,amsmath}
\usepackage{enumerate}
\usepackage{hyperref}
\usepackage{cite}
\usepackage[capitalize]{cleveref}
\usepackage{enumitem}
\usepackage{color}
\usepackage{cancel}
\usepackage{pgf}
\usepackage{svg}
\usepackage{pgfplots}
\usepackage{import}

\usepackage[utf8]{inputenc} 
\usepackage[T1]{fontenc}    
\usepackage{hyperref}       
\usepackage{url}            
\usepackage{booktabs}       
\usepackage{amsfonts}       
\usepackage{nicefrac}       
\usepackage{microtype}
\usepackage{multirow}
\usepackage{xfrac}
\usepackage{todonotes}
\usepackage{titlesec}
\usepackage{caption}

\usepackage{multirow}

\usepackage{caption}
\usepackage{float}

\titleformat{\subsubsection}[runin]
{\normalfont\normalsize\bfseries}{\thesubsubsection}{1em}{}

\usepackage{graphicx}

\usepackage{appendix}
\numberwithin{equation}{section}

\newcommand{\inclu}[0] {\ar@{^{(}->}}

\newcommand{\R}{\mathbb{R}}

\newcommand{\Tr}{\text{Tr}}

\newcommand{\prox}{\text{prox}}

\newcommand{\argmin}{\operatornamewithlimits{argmin}}

\newtheorem{thm}{Theorem}[section]
\newtheorem{theorem}{Theorem}[section]

\newtheorem{proposition}[thm]{Proposition}

\newtheorem{lemma}[thm]{Lemma}

\newtheorem{corollary}[thm]{Corollary}

\newtheorem{conjecture}[thm]{Conjecture}

\crefname{claim}{claim}{claims}
\Crefname{claim}{Claim}{Claims}
\crefname{lem}{lemma}{lemmas}
\Crefname{lem}{Lemma}{Lemmas}
\crefname{algorithm}{algorithm}{algorithms}
\Crefname{algorithm}{Algorithm}{Algorithms}

\theoremstyle{definition}
\newtheorem{remark}{Remark}

\theoremstyle{definition}

\theoremstyle{definition}
\newtheorem*{claim*}{Claim}

\usepackage{mathtools}

\usepackage[algo2e, boxruled]{algorithm2e}

\usepackage[T1]{fontenc}
\usepackage{lmodern}
\crefname{figure}{Figure}{Figures}

\newcommand{\xStar}{x_\star}

\newcommand{\xSpecial}{x_\sigma}
\newcommand{\idxSpecial}{\sigma}

\newcommand{\fSpecial}{f_\sigma}
\newcommand{\gSpecial}{g_\sigma}

\newcommand{\cHat}{\hat{c}}
\newcommand{\cCrit}{c_\mathrm{crit}}

\newcommand{\cCritGrad}{c_{\mathrm{crit,grad}}}
\newcommand{\cOptGrad}{c_{\mathrm{opt,grad}}}

\newcommand{\detDenom}[1]{2N{#1}+4N{#1}i-2i^2+1}
\newcommand{\psiNum}[1]{({#1}-1)(2N{#1}-2N+1)(2N{#1}+1)}
\newcommand{\psiL}{\psi_{\ell}}

\newcommand{\psiU}{\psi_{u}}
\newcommand{\psiLHat}{\hat{\psi}_\ell}
\newcommand{\psiUHat}{\hat{\psi}_u}
\newcommand{\harmonicUpperBd}{\log N + \frac{1}{2N} + \gamma}
\newcommand{\harmonicLowerBd}{\log N + \frac{4N-1}{8N^2} + \gamma}

\newcommand{\cOpt}{c_{\text{opt}}}

\newcommand{\subVar}{s}
\newcommand{\subVarCrit}{\subVar_{\text{crit}}}
\newcommand{\sumH}[1]{\sum_{i=0}^{#1} h_i}
\newcommand{\sumSigH}{\sum_{j=0}^N \sigma_j \sum_{i=0}^{j-1} h_i}

\newcommand{\dimension}{m}

\newcommand{\g}[1]{\mathbf{g_{#1}}}
\renewcommand{\gg}[2]{\g{#1}\g{#2}^T + \g{#2}\g{#1}^T}
\newcommand{\gG}[1]{\g{#1}\g{#1}^T}

\newcommand{\xz}{\mathbf{x_0}}

\pgfplotsset{compat=1.18}

\begin{document}

    \title{On Averaging and Extrapolation for Gradient Descent}

    \author{Alan Luner\footnote{Johns Hopkins University, Department of Applied Mathematics and Statistics, \url{aluner1@jhu.edu}} \qquad Benjamin Grimmer\footnote{Johns Hopkins University, Department of Applied Mathematics and Statistics, \url{grimmer@jhu.edu}}}

	\date{}
	\maketitle

	\begin{abstract}
        This work considers the effect of averaging, and more generally extrapolation, of the iterates of gradient descent in smooth convex optimization. After running the method, rather than reporting the final iterate, one can report either a convex combination of the iterates (averaging) or a generic combination of the iterates (extrapolation). For several common stepsize sequences, including recently developed accelerated periodically long stepsize schemes, we show averaging cannot improve gradient descent's worst-case performance and is, in fact, strictly worse than simply returning the last iterate. In contrast, we prove a conceptually simple and computationally cheap extrapolation scheme strictly improves the worst-case convergence rate: when initialized at the origin, reporting $(1+1/\sqrt{16N\log(N)})x_N$ rather than $x_N$ improves the best possible worst-case performance by the same amount as conducting $O(\sqrt{N/\log(N)})$ more gradient steps. Our analysis and characterizations of the best-possible convergence guarantees are computer-aided, using performance estimation problems. Numerically, we find similar (small) benefits from such simple extrapolation for a range of gradient methods.
	\end{abstract}

    \section{Introduction} \label{Sect:Intro}
There is a long history of using strategic weightings of iterates to produce better solutions in convex optimization. In many nonsmooth optimization settings, returning the final iterate is known to provide provably suboptimal worst-case guarantees~\cite{NonsmoothLastIterate}. A common and well-studied resolution to this problem is to output instead an average (convex combination) of the iterates~\cite{grimmer2024primaldual,Gustavsson2015,NonsmoothOpt_Lan,Shamir2013}. In smooth optimization, on the other hand, averaging is not as commonplace; instead, one often simply returns the final iterate. One intuitive justification for not averaging is that most theory guarantees some monotone decrease, ensuring the last iterate outperforms the rest, e.g., decreasing objective values for gradient descent; however, monotonicity alone is not enough to reach this conclusion.\footnote{One can devise simple examples in which a monotone decrease is guaranteed and yet averaging is beneficial. For example, consider gradient descent with stepsizes $h=2-\epsilon$. Although this is guaranteed to decrease the objective value at each step, its worst-case convergence is attained on a quadratic function $f(x) = \frac{1}{2} x^2$ for which a final averaging step would improve the solution quality.} Beyond the convex combinations underlying averaging, one could consider extrapolations from the iterates, utilizing a not necessarily convex combination of them. Each step of a momentum method, like Polyak's heavy ball method, can be viewed as extrapolating from recent steps to a (hopefully) better next iterate. The use of more sophisticated extrapolation techniques, like Anderson acceleration~\cite{AndersonAccel,AndersonAccel_Recent}, have also proven useful in smooth optimization.

Basic questions remain about the importance (or lack thereof) of averaging and extrapolation in smooth optimization. Our primary focus in this paper is on smooth optimization via gradient descent. Extensions to other accelerated gradient methods will be considered in our numerics in \cref{Sect:Numerics}. For a convex function $f :\R^\dimension \to \R$ with $L$-Lipschitz gradient, gradient descent iterates
\begin{equation}\label{Eqn:GDStep}
    x_{k+1} = x_k - \frac{h_k}{L}\nabla f(x_k) \ ,
\end{equation}
given initial point $x_0$ and a pre-determined stepsize sequence $h = (h_0, \dots, h_{N-1})$. Existing convergence rate theory can handle gradient descent with stepsizes $h_k$ {\it constant} in $(0,2)$, $h_k$ {\it dynamically increasing} up towards two~\cite{Teboulle_4T_Analytical}, and $h_k$ as a special {\it silver} sequence of long stepsizes, frequently greatly exceeding length two~\cite{Silver_Accel}. Existing worst-case convergence analysis for all of these stepsize patterns applies to the final iterate, not utilizing any averaging. This motivates our first major question:
\begin{quote}
    \emph{Can averaging improve the worst-case performance guarantees of gradient descent (in objective gap or gradient size) for any common stepsizes choices (Constant, Dynamically Increasing, or Silver)?}
\end{quote}
We prove it cannot. See the informal statement of this result in \cref{Thm:IntroThmAvg} and \cref{Sect:Averaging} where this claim is formalized and proven.

The fact that averaging cannot help gradient descent motivates the consideration of reporting an extrapolation. Rather than returning a convex combination of the iterates, one could return a linear combination of gradient descent iterates. Our second major question is then:
\begin{quote}
    \emph{Can extrapolating improve the worst-case performance guarantees of gradient descent?}
\end{quote}
We prove it can and show, in fact, a very simple, computationally cheap extrapolation approach suffices. We consider the following intuitive extrapolation strategy. After $N$ steps of gradient descent, instead of returning $x_N$, report
\begin{equation*}
    x_0 + c(x_N-x_0)
\end{equation*}
for some fixed extrapolation factor $c>1$.
This amounts to extrapolating from $x_N$ away from $x_0$. This is natural as one may suspect the direction of progress from the initial iterate to the terminal iterate to ``roughly'' point towards optimal. With $x_0$ chosen as the origin, this simplifies to reporting $cx_N$, adding effectively no additional computational or memory storage costs. We exactly characterize the resulting improvement in worst-case performance guarantee under proper selection of the extrapolation factor $c$.
See the informal statement of this result in \cref{Thm:IntroThmExtrap} and \cref{Sect:SimpleExtrap} where this claim is formalized and proven.

To provide such exact descriptions of worst-case algorithm performance, we turn to performance estimation. The Performance Estimation Problem (PEP), first proposed by Drori and Teboulle \cite{FirstPEP}, established a new perspective framing a method's worst-case performance as an optimization problem. Building on this initial proposal, Taylor et al.~\cite{Interpolation} showed the PEP search for a worst-case problem instance is equivalent to a Semidefinite Program (SDP). This approach to optimization problems is a doubly important tool. First, PEP serves as a black box to numerically compute worst-case convergence rates under varied algorithmic configurations. Observing these numerical results can help identify patterns and shortcomings/slackness in current performance guarantees. Second, observing the exact worst-case problem instances and corresponding dual optimality certificates produced by PEP can often translate into proving new and improved convergence guarantees. For readers unfamiliar with PEP, we refer to the library and resources of~\cite{pepit2022} and the foundational early paper by Taylor et al.~\cite{Interpolation2}.

Particularly relevant to our investigation, the PEP framework has seen substantial use in the design of stepsizes for gradient descent. In their paper introducing PEP, Drori and Teboulle proved a tight $O(1/N)$ convergence bound \cite[Theorem 3.1]{FirstPEP} for gradient descent with constant stepsize $h\in(0,1]$. They also conjectured a similar bound \cite[Conjecture 3.1]{FirstPEP} for $h\in (1,2)$ based on strong numerical evidence from performance estimation. More recently, by solving a related nonconvex problem via branch-and-bound, Das Gupta et al.~\cite{BnB} numerically identified ``globally optimal'' stepsize patterns for fixed small $N\leq 25$, which were often much larger than length two. Through their general branch-and-bound method, one can optimize an algorithm's hyperparameters, such as the stepsize sequence or averaging weights, for any reasonably small fixed number of iterations; this ability to answer questions of global optimality offers a useful comparison for our later results. Motivated by these numerical results, Grimmer et al.~\cite{Grimmer_LongStep,Grimmer_Accel} used the PEP framework to analyze specific ``straightforward'' sequences with frequent long steps ($h_k>2$), eventually showing a slightly accelerated convergence rate of $O\left(1/N^{1.0564}\right)$. Concurrently,  Altschuler and Parrilo~\cite{Silver_Accel} following up on the earlier Master's thesis work~\cite{altschuler2018greed}, introduced the \textit{silver} stepsize sequence for gradient descent, which attains a stronger rate of $O\left(1/N^{1.2716}\right)$.

\subsection{Our Contributions}

Our two primary contributions show that averaging does not improve (and, in fact, strictly worsens) the worst-case convergence guarantees for gradient descent under common stepsize sequences, while a very simple extrapolation can strictly improve gradient descent performance. Our analysis in both settings is tight, having exactly matching worst-case problem instances derived from associated PEP solutions. We measure solution quality both by function value gap and gradient norm. These two main results are informally stated below with $\Delta_N$ denoting the standard simplex in $\mathbb{R}^N$ and $\mathfrak{F}_{L,D}$ denoting the set of all considered smooth convex problem instances: namely, all pairs $(x_0,f)$ with $f$ convex and $L$-smooth and $x_0$ at most distance $D$ from a minimizer of $f$.

\begin{theorem}[Informal, The Optimal Averaging is to Not Average] \label{Thm:IntroThmAvg}
    The optimal averaging of the iterates of gradient descent is to return the final iterate when the stepsizes $h_k$ are chosen as any of
    \begin{enumerate}
        \item[(i)] constant $h\in(0,1]$,
        \item[(ii)] constant $h\in (1,2)$ (assuming $N$ sufficiently large and \cite[Conjecture 3.1]{FirstPEP} holds),
        \item[(iii)] dynamically increasing as proposed by Teboulle and Vaisbourd~\cite{Teboulle_4T_Analytical},
        \item[(iv)] silver long steps as proposed by Altschuler and Parrilo~\cite{Silver_Accel} (assuming $N$ is one less than a power of two and our \cref{Conj:SilverRate} holds).
    \end{enumerate}
    That is, for the stepsize sequences above, $\sigma = (0,\dots,0,1)$ is the unique minimizer of
    \begin{equation}\tag{\ref{Eqn:minMaxFVal}}
        \min_{\sigma\in\Delta_{N+1}} \max_{(x_0,f)\in \mathfrak{F}_{L,D}} f\left(\sum^{N}_{j=0} \sigma_j x_j\right) - \inf f = \frac{LD^2}{4\sumH{N-1}+2}
    \end{equation}
    and
    \begin{equation} \tag{\ref{Eqn:minMaxGradNorm}}
        \min_{\sigma\in\Delta_{N+1}} \max_{(x_0,f)\in \mathfrak{F}_{L,D}} \left\|\nabla f\left(\sum^{N}_{j=0} \sigma_j x_j\right)\right\| =\frac{LD}{\sumH{N-1}+1} \ .
    \end{equation}
    
\end{theorem}

\begin{remark} \label{Rem:SilverProven}
    After the initial release of this manuscript, two preprints claimed to respectively prove both conjectures referenced above. Kim ~\cite{Kim_ProofOfConstantStepGD} presents a proof of \cite[Conjecture 3.1]{FirstPEP} and Wang et al.~\cite{Wang_ProofOfConjecture} present a proof of \cref{Conj:SilverRate}. Given these results, the statement of \cref{Thm:IntroThmAvg} holds for silver stepsize patterns and for appropriate\footnote{Specifically, for $h$ less than or equal to the solution to the equation $\frac{1}{2Nh+1} = (1-h)^{2N}$.} constant stepsizes $h \in (1,2)$.
\end{remark}

\begin{theorem}[Informal, Strict Gain of $\sqrt{N/\log(N)}$ from Simple Extrapolation] \label{Thm:IntroThmExtrap}
    For any constant stepsize $h \in (0,1]$ and simple extrapolation factor $c$ up to size $1 + O\left(\frac{1}{\sqrt{N \log N}}\right)$, gradient descent has
    \begin{equation} \label{Eqn:IntroExtrapBoundGeneral}
        \max_{(x_0,f)\in \mathfrak{F}_{L,D}} f(x_0 + c(x_N-x_0)) - \inf f  = \frac{LD^2}{4Nhc+2} \ . 
    \end{equation}
    For example, setting $x_0=0$ without loss of generality,
    \begin{equation} \label{Eqn:IntroExtrapBoundSpec}
        \max_{(x_0,f)\in \mathfrak{F}_{L,D}} f\left(\left(1 + \frac{1}{4\sqrt{N\log(N)}}\right)x_N\right) - \inf f  = \frac{LD^2}{4Nh + \sqrt{\frac{N}{\log(N)}}h + 2} \ . 
    \end{equation}
\end{theorem}

Note that gradient descent is not optimal among all first-order methods for smooth convex optimization as accelerated methods provide faster rates of $O(1/N^2)$ \cite{OGM,Nesterov}. As a result, our theory should not be viewed as advancing the state of the art in performance for this general class. Instead, gradient descent provides a structured and fundamental smooth optimization setting where we can provide exact proof that averaging offers no gains and quantify the provable gains simple extrapolation can provide. This simple extrapolation may have practical relevance in settings where memory is the dominant computational constraint (See \cref{Remark:memoryCost}). It can be easily implemented at the end of any first-order method, with no a priori knowledge of the number of steps to be taken. To extend our main results beyond gradient descent, we provide numerical characterizations of the worst-case performance of several common accelerated methods with extrapolation. We find nearly universal improvements (small but nonzero) in the worst-case theoretical guarantees from simple extrapolations. Extensions to projected and proximal gradient methods are also discussed when applicable.

\paragraph{Outline.} First, \cref{Sect:Prelim} introduces our notation and the performance estimation problem framework upon which our results are primarily built. \cref{Sect:Averaging} then proves our negative results regarding the use of averaging in gradient descent (for any of a range of common stepsizes). \cref{Sect:SimpleExtrap} proves our positive result on the benefits of even very simple extrapolation on worst-case performance. Finally, \cref{Sect:Numerics} presents numerical results strongly indicating the degree to which our insights from gradient descent generalize. A few algebraic simplifications are deferred to the associated \texttt{Mathematica}~\cite{Mathematica} notebook available at \href{https://github.com/alanluner/GDAvgExtrap}{\texttt{github.com/alanluner/GDAvgExtrap}}.
    \section{Preliminaries and Performance Estimation Problems} \label{Sect:Prelim}
In this section, we introduce Performance Estimation Problems (PEP) and the specific problems relevant to our analysis. Consider a smooth convex minimization problem of the form
\begin{equation}
    \min_{x \in \mathbb{R}^\dimension} f(x)
\end{equation}
where $f:\mathbb{R}^\dimension \rightarrow \mathbb{R}$ is $L$-smooth (i.e. $\nabla f$ is $L$-Lipschitz) and convex. We assume a minimizer $\xStar$ of $f$ exists and suppose a point $x_0$ is known with $\|x_0 - \xStar\| \leq D$ for some $D \in \mathbb{R}_+$. 
Throughout, we denote the Euclidean inner product by $\langle\cdot,\cdot\rangle$, and all norms are the associated two-norm. Note that a differentiable function $f$ is $L$-smooth and convex if and only if
\begin{equation}\label{eq:smoothness-charcterization}
    f(y) \geq f(x) + \langle \nabla f(x), y-x\rangle + \frac{1}{2L}\|\nabla f(x) - \nabla f(y)\|^2 \qquad \forall x,y\in\mathbb{R}^\dimension \ .
\end{equation}
One particularly useful smooth convex function is the ``Huber'' function, which often occurs as a worst-case problem for smooth optimization, as highlighted in \cite{Interpolation} and many related works \cite{FirstPEP,OGM,Interpolation2,Teboulle_4T_Analytical}. We denote this simple one-dimensional function by
\begin{equation}\label{Eqn:Huber}
    \phi_{L,\eta}(x) := \begin{cases}
        \frac{L}{2}x^2 \quad &\text{if }|x| \leq \eta \\
        L\eta|x|-\frac{L\eta^2}{2} \quad &\text{if } |x| > \eta  \ .
    \end{cases}
\end{equation}

We primarily consider applying gradient descent by iterating as in \eqref{Eqn:GDStep} using a pre-determined stepsize sequence $h=(h_0,\dots,h_{N-1})$.
Rather than simply reporting the terminal iterate $x_N$, we consider reporting an averaged/extrapolated point $\xSpecial$, defined by
\begin{equation}\label{Eqn:xSpecialDef}
    \xSpecial = \sum_{j=0}^{N} \sigma_j x_j
\end{equation}
given weights $\sigma = (\sigma_0, \dots, \sigma_N)$. Note no modification is made to the iterates of gradient descent; the only change is to which point is ultimately reported and, hence, where performance (in the objective gap or gradient norm) is measured.

We broadly refer to this scheme as \textit{general extrapolation}. As a specific case, we define \textit{iterate averaging} as any choice of $\sigma$ such that $\sigma \geq 0$ and $\sum_{j=0}^N \sigma_j = 1$. We will say that an averaging scheme $\sigma$ is \textit{nondegenerate} provided that $\sum_{j=0}^N \sigma_j x_j \neq x_N$ (i.e. $\sigma \neq (0,\dots,0,1)$). Observe that our framing of $\xSpecial$ is equivalent to 
\begin{align}\label{Eqn:xSpecialGradForm}
    \xSpecial &= \sum_{j=0}^N \sigma_j x_j = \sum_{j=0}^N \sigma_j \left(x_0 - \frac{1}{L} \sum_{k=0}^{j-1} h_k \nabla f(x_k) \right) \\
    &= x_0 - \frac{1}{L} \sum_{k=0}^{N-1} \left(h_k \sum_{j=k+1}^N \sigma_j\right) \nabla f(x_k) \ . \nonumber
\end{align}
So a general extrapolation can be denoted by $x_0 + \text{span} \{ \nabla f(x_0), \dots, \nabla f(x_{N-1}) \}$.

\subsection{Performance Estimation Problems with Averaging/Extrapolation}
When evaluating the convergence of an optimization method, there are several different performance measures commonly considered. We will introduce the PEP framework using the objective gap $f(x) - f(\xStar)$ as the measure of algorithm performance. Although this is done to ease our initial development, we will also present results on guaranteeing a small gradient norm, $\|\nabla f (x)\|$.
In particular, our primary focus is to compare the worst-case performance of an averaged or extrapolated point $\xSpecial$ with that of the last iterate $x_N$.

One can frame deriving a convergence guarantee for an algorithm in terms of understanding its performance on a worst-case problem instance. From this perspective, worst-case analysis is an optimization problem: find the problem instance with the maximum final/reported objective gap, gradient norm, etc. We define our worst-case performance $p_{N,L,D}(\sigma,h)$ by
\begin{equation}
    p_{N,L,D}(\sigma,h) := \begin{cases} \max_{x_0,\xStar,f} \quad  & f(\xSpecial) - f(\xStar) \\
                    \text{s.t.} \quad & f(x)\geq f(y)+\langle \nabla f(y),x-y\rangle +\frac{1}{2L}\|\nabla f(x) - \nabla f(y)\|^2 \quad \forall x,y \\
                    &\|x_0-\xStar\| \leq D \\
                    &\nabla f(\xStar) = 0 \\
                    &x_{k+1} = x_k - \frac{h_k}{L}\nabla f(x_k) \quad \quad k=0,\dots,N-2\\
                    &\xSpecial = x_0 - \frac{1}{L} \sum_{k=0}^{N-1} \left(h_k \sum_{j=k+1}^N \sigma_j\right) \nabla f(x_k)
                \end{cases}
\end{equation}
where our first constraint comes from the standard identity for $L$-smooth, convex functions~\eqref{eq:smoothness-charcterization}. We use the alternate form \eqref{Eqn:xSpecialGradForm} of $\xSpecial$ to remove its dependence on $x_N$. Consequently, we can exclude $x_N$ from our problem formulation. Equivalently, one could view $\xSpecial$ as a modified replacement of $x_N$.

Finite dimensional relaxations of this formulation were first considered by Drori and Teboulle~\cite{FirstPEP}. Subsequently and quite surprisingly, the Interpolation Theorem of Taylor et al.~\cite{Interpolation} established an equivalent finite-dimensional problem one could consider. Rather than enforce our constraints for all points in our domain, it is only necessary to enforce them along the points of interest in our algorithm. We denote $f_k = f(x_k)$, $g_k = \nabla f(x_k)$, and $\mathcal{I}=\{\star,0,1,\dots,N-1,\idxSpecial\}$, where we abuse notation to allow $\star$ and $\sigma$ to act as indices ($f_\star = f(x_\star), f_\sigma = f(x_\sigma)$, etc.). Then our problem becomes

\begin{equation}\label{Eqn:Interpolation}
    p_{N,L,D}(\sigma,h) = \begin{cases} \max_{x_0,\xStar,f} \quad & f_\idxSpecial - f_\star \\
                    \text{s.t.} \quad &f_i\geq f_j+\langle g_j,x_i-x_j\rangle +\frac{1}{2L}\| g_i - g_j\|^2 \quad \forall i\neq j \in \mathcal{I} \\
                    &\|x_0-\xStar\| \leq D \\
                    &g_\star = 0 \\
                    &x_{k+1} = x_k - \frac{h_k}{L} g_k \quad \quad k=0,\dots,N-2\\
                    & \xSpecial = x_0 - \frac{1}{L} \sum_{k=0}^{N-1} \left(h_k \sum_{j=k+1}^N \sigma_j\right) g_k \ .
                \end{cases}
\end{equation}
By translation, we can also fix $\xStar=0$ and $f_\star=0$ without loss of generality.

Finally, following the methods in many previous works \cite{Drori2019,Interpolation,Interpolation2}, we slightly relax our discrete problem \eqref{Eqn:Interpolation} to form a solvable SDP. We adopt the notation used in \cite{BnB} and introduce it here. We define
\begin{align*}
    & F = [f_0|f_1|\dots|f_{N-1}|\fSpecial] \in \R^{1 \times(N+1)} \\
    & H = [x_0|g_0|g_1|\dots|g_{N-1}|\gSpecial] \in \R^{\dimension \times(N+2)} \\
    & G = H^T H \in \mathbb{S}_+^{N+2} \ .
\end{align*}
We also define special vectors for selecting particular elements of our matrices using standard basis vectors $e_i$ (the corresponding space for $e_i$ should be clear from context if not specified):
\begin{align*}
    & \mathbf{g_\star} = 0 \in \R^{N+2}, & \mathbf{g_i} = e_{i+2} \in \R^{N+2} \quad i=0,\dots, N-1 \\
    & \mathbf{f_\star} = 0 \in \R^{N+1}, & \mathbf{f_i} = e_{i+1} \in \R^{N+1} \quad i=0,\dots, N-1 \\
    & \mathbf{\xStar} = 0 \in \R^{N+2}, & \mathbf{x_{i+1}} = \mathbf{x_i}-\frac{h_i}{L}\mathbf{g_i} \quad  i=0,\dots,N-2 \\
    & \mathbf{x_0} = e_1 \in \R^{N+2}
\end{align*}
and to account for $\xSpecial$, we define
\begin{align*}
    & \mathbf{g_\sigma} = e_{N+2} \in \R^{N+2} \\
    & \mathbf{x_\sigma} = \mathbf{x}_0 - \frac{1}{L} \sum_{k=0}^{N-1} \left(h_k \sum_{j=k+1}^N \sigma_j\right) \mathbf{g_k} \\
    & \mathbf{f_\sigma} = e_{N+1} \in \R^{N+1} \ .
\end{align*}
Through this construction, we have encoded the gradient steps of the algorithm into our matrices $F$, $G$, and $H$; we have $x_i=H\mathbf{x_i}$, $g_i = H\mathbf{g_i}$, and $f_i = F\mathbf{f_i}$ for all $i \in \mathcal{I}$. Next, using the symmetric outer product $x \odot y = \frac{1}{2}(xy^T + yx^T)$, define
\begin{align*}
    & A_{i,j}(h) = \mathbf{g_j} \odot (\mathbf{x_i} - \mathbf{x_j}) \in \mathbb{S}^{N+2} \\
    & B_{i,j}(h) = (\mathbf{x_i} - \mathbf{x_j})\odot(\mathbf{x_i} - \mathbf{x_j}) \in \mathbb{S}^{N+2}\\
    & C_{i,j} = (\mathbf{g_i} - \mathbf{g_j})\odot (\mathbf{g_i} - \mathbf{g_j}) \in \mathbb{S}^{N+2}\\
    & a_{i,j} = \mathbf{f_j} - \mathbf{f_i} \in \R^{N+1}
\end{align*}
for all $i,j \in \mathcal{I}$, with $i\neq j$.
These matrices satisfy the useful identities
\begin{align*}
        &\langle g_j, x_i-x_j \rangle = \Tr GA_{i,j}(h)\\
        &\|x_i-x_j\|^2=\Tr GB_{i,j}(h)\\
        &\|g_i-g_j\|^2=\Tr GC_{i,j} \ .
\end{align*}
These definitions enable a succinct statement of the performance estimation problem as an SDP with a rank constraint:
\begin{equation}
    p_{N,L,D}(\sigma,h) = \begin{cases} \max_{F,G} \quad & F a_{\star,\idxSpecial} \\
                    \text{s.t. } & Fa_{i,j}+\Tr G A_{i,j}(h) + \frac{1}{2L}\Tr G C_{i,j} \leq 0   \quad \forall i\neq j \in \mathcal{I}\\
                    & G \succeq 0 \\
                    & \Tr G B_{0,\star} \leq D^2  \\
                    & \mathrm{rank} G \leq \dimension \ .
                \end{cases}
\end{equation}
Applying the rank constraint ensures that there is a solution to $G = H^T H$ for $H \in \R^{\dimension \times (N+2)}$. If we remove this rank constraint, we obtain an upper bounding SDP

\begin{equation}\label{Eqn:primal}
    p_{N,L,D}(\sigma,h) \leq \begin{cases} \max_{F,G} \quad & F a_{\star,\idxSpecial} \\
                    \text{s.t. } & Fa_{i,j}+\Tr G A_{i,j}(h) + \frac{1}{2L}\Tr G C_{i,j} \leq 0   \quad \forall i\neq j \in \mathcal{I}\\
                    & G \succeq 0 \\
                    & \Tr G B_{0,\star} \leq D^2 \ . \\
                \end{cases}
\end{equation}
This SDP can be made equivalent by applying an additional assumption that the problem dimension $\dimension$ is at least $N+2$ \cite{Interpolation}. However, our analysis will only rely on upper bounds, so this relaxation is sufficient, and we do not require this assumption on problem dimension. This completes our derivation of the PEP SDP for the averaged/extrapolated performance measure $f(\xSpecial)-f(\xStar)$. Other common performance measures would follow a very similar derivation. Going forward, we simplify our notation to
\begin{equation}
p(\sigma) := p_{N,L,D}(\sigma,h) \ ,
\end{equation}
but note that our solution is a function of each of those now hidden parameters. We especially emphasize the role of $N$ as a known and fixed parameter in the optimization problem.

Finally, we define the dual SDP in preparation for our use of the dual certificate in our later proofs. Introducing dual variables $\lambda_{i,j} \in \R$, $v \in \R$, and $Z \in \mathbb{S}^{N+2}$, we let
\begin{equation} \label{Eqn:dual}
    d(\sigma) := \left\{\begin{array}{lll}
        \min_{v,\lambda,Z} \quad  & v D^2 \\
        \text{s.t. } & \sum_{i\neq j} \lambda_{i,j} a_{i,j} - a_{\star,\idxSpecial} = 0 \\
        & v B_{0,\star} + \sum_{i\neq j} \lambda_{i,j} (A_{i,j}(h) + \frac{1}{2L}C_{i,j}) = Z  \\
        & Z \succeq 0 \\
        & v \geq 0, \lambda_{i,j} \geq 0  \quad \forall i\neq j \in \mathcal{I}
    \end{array}\right.
\end{equation}
with the role of $\sigma$ implicit in the various vector and matrix definitions.
This SDP \eqref{Eqn:primal} and its dual \eqref{Eqn:dual} provide a widely useful framework for analyzing worst-case problem instances~\cite{altschuler2018greed,BnB,FirstPEP,Grimmer_LongStep,Grimmer_Accel,Interpolation}. This dual problem will play a central role in proving our new convergence guarantees as each feasible point to the dual problem constitutes a proof of an upper bound on $p(\sigma)$.

    \section{Averaging is Strictly Worse than the Last Iterate}\label{Sect:Averaging}

In this section, we address our first major question regarding the effect of iterate averaging on gradient descent's convergence guarantees. Our approach to establishing that averaging cannot benefit (and in fact harms) worst-case convergence guarantees works by first observing a common structure in the tight convergence bounds for many common stepsize selections (\cref{SubSec:Tight-Bounds}) and second, showing any stepsize with this worst-case structure cannot benefit from averaging (\cref{SubSect:AvgTheory}).

\subsection{Common Structure Among Tight Last Iterate Convergence Guarantees} \label{SubSec:Tight-Bounds}
The performance estimation framework has enabled the development of exactly tight convergence rates for gradient descent. That is, convergence guarantees that are attained with equality by some problem instance, establishing that no better guarantee is possible. Below, we summarize the known (or conjectured) tight convergence rates for four different families of stepsizes $h_k$.
For each considered stepsize policy on any $L$-smooth convex $f$ with $\|x_0-\xStar\|\leq D$, these tight rates for gradient descent take the form
\begin{equation}\label{Eqn:ObjGapRate}
    f(x_N)-f(\xStar) \leq \frac{LD^2}{4\sum_{i=0}^{N-1} h_i +2}
\end{equation}
and
\begin{equation}\label{Eqn:GradNormRate}
    \|\nabla f(x_N)\| \leq \frac{LD}{\sum_{i=0}^{N-1} h_i +1} \ .
\end{equation}
Before discussing particular stepsize selections, we first note that no stepsize sequence $h_k$ can produce a rate faster than those above. This follows from considering simple Huber functions as prior works~\cite{FirstPEP,OGM,Interpolation,Interpolation2,Teboulle_4T_Analytical} have identified as common worst-case problem instances. While we do not focus on momentum methods in this section, we note that the prevalence of the Huber function appears to be shared for momentum methods such as Nesterov acceleration and OGM, as conjectured in \cite{Interpolation}.
\begin{lemma}\label{Lem:Huber-Tightness}
    Consider gradient descent with stepsizes $h = (h_0, \dots, h_{N-1})$. For any $L,D>0$, there exists a problem instance with $L$-smooth convex $f$ and $\|x_0-\xStar\|\leq D$ such that
    $$ f(x_N)-f(\xStar) = \frac{LD^2}{4\sum_{i=0}^{N-1} h_i +2} \ . $$
    There also exists an $L$-smooth convex $f$ with
    $$ \|\nabla f(x_N)\| = \frac{LD}{\sum_{i=0}^{N-1} h_i +1} \ . $$
\end{lemma}
\begin{proof}
    In both cases, this problem instance takes the form of $x_0=D$ and $f = \phi_{L,\eta}$ as a Huber function. Note that provided $\eta \leq D/(\sum_{i=0}^{N-1} h_i+1)$, the first $N$ iterates of gradient descent all remain larger than $\eta$, being given by $x_k = D - \eta \sum_{i=0}^{k-1} h_i$ and having constant gradient $\nabla f(x_k) = \eta L$.
    To attain equality in the objective gap, selecting $\eta = D/(2\sum_{i=0}^{N-1} h_i+1)$ has $f(x_N)-f(\xStar) = \frac{LD^2}{4\sum_{i=0}^{N-1} h_i +2}$.     
    To attain equality in the gradient norm, selecting $\eta = D/(\sum_{i=0}^{N-1} h_i+1)$ has $\|\nabla f(x_N)\| = \frac{LD}{\sum_{i=0}^{N-1} h_i +1}$.
\end{proof}

As a result, guarantees of the form~\eqref{Eqn:ObjGapRate} and~\eqref{Eqn:GradNormRate} holding imply more strongly that
\begin{equation} \label{Eqn:ConditionFVal}
    \max_{(x_0,f)\in \mathfrak{F}_{L,D}} f(x_N) - \inf f = \frac{LD^2}{4\sumH{N-1} + 2} \tag{C1}
\end{equation}
and
\begin{equation}\label{Eqn:ConditionGradNorm}
    \max_{(x_0,f)\in \mathfrak{F}_{L,D}} \| \nabla f(x_N) \| = \frac{LD}{\sumH{N-1}+1} \tag{C2}
\end{equation}
hold, respectively. We show below in \cref{Thm:AvgFunc,Thm:AvgGradNorm} that these two conditions imply averaging is strictly worse than returning the last iterate in terms of worst-case guarantees.

\paragraph{Constant Stepsizes $h_k = h \in (0,1]$} Perhaps the simplest setting of gradient descent is the use of a constant stepsize $h_k=h\in (0,1]$. The seminal work of~\cite{FirstPEP} and subsequently~\cite{Teboulle_4T_Analytical} showed such stepsizes have guarantees of the form~\eqref{Eqn:ObjGapRate} and~\eqref{Eqn:GradNormRate}, giving bounds of $LD^2/(4Nh + 2)$ and $LD/(Nh+1)$ on objective gap and gradient norm convergence, respectively. Note tightness of these bounds is easily verified by the examples in \cref{Lem:Huber-Tightness}, and so the conditions~\eqref{Eqn:ConditionFVal} and~\eqref{Eqn:ConditionGradNorm} hold. For extensions of these ideas to monotone operators, see~\cite[Remark 4.10]{LiederThesis}.

\paragraph{Constant Stepsizes $h_k = h \in (1,2)$}
When using constant stepsizes $h_k = h \in (1,2)$, the conjectured tight convergence rates from~\cite[Conjecture 3.1]{FirstPEP} and \cite[Conjecture 3]{Interpolation}, respectively, are that
\begin{equation}\label{Eqn:LargeHBound}
    f(x_N) - f(\xStar) \leq \frac{LD^2}{2} \max\left\{\frac{1}{2Nh+1}, |1-h|^{2N}\right\}
\end{equation}
and
\begin{equation}\label{Eqn:LargeHBoundGrad}
    \|f(x_N) \| \leq LD \max \left\{\frac{1}{Nh+1}, |1-h|^N\right\} \ .
\end{equation}
Provided $N$ is large enough relative to the fixed value of $h$, each first case above dominates, and the (conjectured) convergence guarantees take the form~\eqref{Eqn:ObjGapRate} and~\eqref{Eqn:GradNormRate}. Again, tightness and the conditions~\eqref{Eqn:ConditionFVal} and~\eqref{Eqn:ConditionGradNorm} follow from \cref{Lem:Huber-Tightness}. See \cref{Rem:SilverProven} for progress on this conjecture.

\paragraph{Dynamic Stepsizes $h_k \rightarrow 2$}
Note convergence cannot be guaranteed for gradient descent with $h_k$ constant and greater than or equal to two. A sequence of stepsizes $h_k$ approaching this boundary was proposed and analyzed by~\cite{Teboulle_4T_Analytical}. 
They considered
\begin{equation}\label{Eqn:DynamicStepsizes}
    h_k = \frac{-\sumH{k-1} + \sqrt{\left(\sumH{k-1} \right)^2 + 8\left(\sumH{k-1} + 1\right)}}{2}
\end{equation}
with $h_0 = \sqrt{2}$.
Theorem 4 of~\cite{Teboulle_4T_Analytical} established tight convergence guarantees of the common form~\eqref{Eqn:ObjGapRate} and~\eqref{Eqn:GradNormRate}, establishing conditions~\eqref{Eqn:ConditionFVal} and~\eqref{Eqn:ConditionGradNorm} hold.

\paragraph{Silver Stepsizes $h_k$ (often much greater than two)}
Going beyond the limitation of stepsizes being at most length two, Altschuler and Parrilo~\cite{Silver_Accel} considered the ``silver'' stepsize pattern:
Using the silver ratio, $\rho = 1+\sqrt{2}$, the sequence $h^{(N)}$ is defined recursively for any $N = 2^m -1$ by concatenation $h^{(2N+1)} = (h^{(N)},1+\rho^{m-1}, h^{(N)})$ and with $h^{(1)} = \sqrt{2}$. Note this includes arbitrarily large stepsizes, having $h^{(N)}_{2k} \approx k^{1.2716}$ whenever $k$ is a power of two.
Theorem 1.1 of~\cite{Silver_Accel} showed silver stepsizes have an accelerated convergence rate of $f(x_N)-f(\xStar) = O(LD^2/N^{1.2716})$. Numerically computing the worst-case performance of silver stepsizes via PEP, we find for every $N$ one less than a power of two, the convergence matches~\eqref{Eqn:ObjGapRate} and~\eqref{Eqn:GradNormRate}, see \cref{Tbl:SilverNumerics}. This leads us to make the following conjecture.
\begin{conjecture}\label{Conj:SilverRate}
    Consider a gradient descent algorithm of fixed length $N = 2^m - 1$ and the corresponding silver step sequence $h = h^{(N)}$. Then any $L$-smooth convex $f$ has
    \begin{align*}
        f(x_N) - f(\xStar) &\leq \frac{LD^2}{4 \sumH{N-1} + 2} = \frac{LD^2}{4\rho^{\log_2(N+1)} - 2} \ , \\ 
        \|f(x_N) \| &\leq \frac{LD}{ \sumH{N-1} + 1} = \frac{LD}{\rho^{\log_2(N+1)}} \ . 
    \end{align*}
\end{conjecture}
\noindent In the statement above, we use the fact that $\sumH{N-1} = \rho^{\log_2(N+1)} - 1$. If true, tightness of these bounds is immediate from \cref{Lem:Huber-Tightness}, giving conditions~\eqref{Eqn:ConditionFVal} and~\eqref{Eqn:ConditionGradNorm}.
See \cref{Rem:SilverProven} for progress on this conjecture.

\begin{table}
    \centering\footnotesize
    \begin{tabular}{|c|c|c|c|c|c|c|c|}
    \hline
         & $N=1$ & $N=3$ & $N=7$ & $N=15$ & $N=31$ & $N=63$ & $N=127$\\
         \hline
         Silver Obj.~Gap PEP & 0.13060 & 0.04692 & 0.01842 & 0.00747 & 0.00307 & 0.00127 & 0.00052\\
         Difference from~\eqref{Eqn:ConditionFVal} &-8.567e-10 & 1.581e-8 & 5.148e-11 & -2.942e-9 & -5.057e-12 & -3.232e-8 & -1.660e-13\\
         \hline
         Silver Grad.~Norm PEP & 0.41421 & 0.17157 & 0.07107 & 0.02944 & 0.01219 & 0.00505 & 0.00209 \\
         Difference from~\eqref{Eqn:ConditionGradNorm} & -3.669e-14 & -1.976e-13 & -6.178e-12 & -5.630e-9 & -1.204e-8 & -3.685e-10 & -5.238e-9 \\    
         \hline
    \end{tabular}
    \caption{Numerical results from \texttt{Mosek} with feasibility tolerances set as $10^{-12}$ solving the PEP SDP for the worst-case objective gap and gradient norm of $x_N$ after $N$ steps of the silver stepsize sequence. Differences from the value in \cref{Conj:SilverRate} are presented, which remain relatively small.}
    \label{Tbl:SilverNumerics}
\end{table}

\subsection{Suboptimal Convergence from Averaging}\label{SubSect:AvgTheory}
The above tight characterizations (down to the constants) of the worst-case performance of gradient descent's final iterate provide a direct quantity against which to compare the performance of a proposed averaging scheme. We find that the conditions~\eqref{Eqn:ConditionFVal} and~\eqref{Eqn:ConditionGradNorm} imply every nondegenerate iterate averaging $\sigma$ provides a strictly worse final objective value. This is formalized for objective gap convergence in \cref{Thm:AvgFunc} and gradient norm convergence in \cref{Thm:AvgGradNorm}. Similar to \cref{Lem:Huber-Tightness}, the proof of these results is based on considering related Huber functions.

As an immediate consequence of these theorems, averaging cannot benefit (and strictly worsens) the convergence guarantees for gradient descent with $h_k$ constant less than one or as the dynamic sequence in~\eqref{Eqn:DynamicStepsizes}. Further, if $N$ is sufficiently large and the numerically supported conjecture of~\cite[Conjecture 3.1]{FirstPEP} holds (See \cref{Rem:SilverProven}), constant stepsizes between one and two cannot benefit from averaging. Similarly, if $N$ is one less than a power of two and the numerically supported \cref{Conj:SilverRate} holds, the accelerated rate following from silver stepsizes cannot benefit from averaging.
As an aside, numerical PEP evaluations indicate the conditions~\eqref{Eqn:ConditionFVal} and~\eqref{Eqn:ConditionGradNorm} do not appear to hold for the straightforward sequences studied by~\cite{Grimmer_LongStep,Grimmer_Accel}, indicating they may benefit from averaging.

\begin{theorem}\label{Thm:AvgFunc}
    Consider gradient descent with stepsizes $h=(h_0, \dots, h_{N-1}) > 0$. For any $L,D > 0$, if the worst final iterate objective gap is characterized by~\eqref{Eqn:ConditionFVal}, then the optimal averaging is to return the final iterate. That is,
    \begin{equation} \label{Eqn:minMaxFVal}
        \min_{\sigma\in\Delta_{N+1}} \max_{(x_0,f)\in \mathfrak{F}_{L,D}} f\left(\sum^{N}_{j=0} \sigma_j x_j\right) - \inf f = \frac{LD^2}{4\sumH{N-1}+2}
    \end{equation}
    with $\sigma = (0,\dots,0,1)$ as the unique minimizer.
\end{theorem}
\begin{proof}
    To prove this result, it suffices to show that for any nondegenerate $\sigma$, there exists an $L$-smooth convex function $f$ and initialization $\|x_0-\xStar\|\leq D$ such that
    \begin{equation*}
        f(\xSpecial) - f(\xStar) > \frac{LD^2}{4\sumH{N-1} + 2} \ .
    \end{equation*}
    We consider the same Huber function establishing tightness of the last iterate convergence in \cref{Lem:Huber-Tightness}. Let $x_0=D$ and $f=\phi_{L,\eta}$ with $\eta=D/(2\sum_{i=0}^{N-1}h_i +1)$.
    Noting $\eta < 
    D/(\sum_{i=0}^{N-1}h_i +1)$, each $k\leq N$ has $x_k = D-\eta \sumH{k-1} > \eta$ and $\nabla f(x_k) = \eta L$. Observe that $f$ is affine over the convex hull of the iterates (since it is affine on $x\geq \eta$) with
    \begin{align*}
        f(x_0) &= \frac{LD^2}{2\sum_{i=0}^{N-1} h_i +1} - \frac{LD^2}{2(2\sum_{i=0}^{N-1} h_i +1)^2} \ , \\
        f(x_N) &= \frac{LD^2}{2(2\sum_{i=0}^{N-1} h_i +1)} < f(x_0) \ .
    \end{align*}
    Hence $f(x_N) < f(x_k)$ for all $k<N$. Then for any nondegenerate $\sigma$, this affine property ensures
    \begin{align*}
        f(\xSpecial) - f(\xStar) = f(\xSpecial) & = \sum_{j=0}^N \sigma_j f(x_j)
         > f(x_N) 
         = \frac{LD^2}{4\sumH{N-1}+2}
    \end{align*}
    Therefore, the degenerate averaging $\sigma = (0,\dots,0,1)$ uniquely minimizes the worst-case objective gap $\max_{(x_0,f)} f(\xSpecial) - \inf f$.
\end{proof}

\begin{theorem}\label{Thm:AvgGradNorm}
    Consider gradient descent with stepsizes $h=(h_0, \dots, h_{N-1}) > 0$. For any $L,D > 0$, if the worst final iterate gradient norm is characterized by~\eqref{Eqn:ConditionGradNorm}, then the optimal averaging is to return the final iterate. That is,
    \begin{equation}\label{Eqn:minMaxGradNorm}
        \min_{\sigma\in\Delta_{N+1}} \max_{(x_0,f)\in \mathfrak{F}_{L,D}} \left\|\nabla f\left(\sum^{N}_{j=0} \sigma_j x_j\right) \right\| = \frac{LD}{\sumH{N-1}+1}
    \end{equation}
    with $\sigma = (0,\dots,0,1)$ as the unique minimizer.
\end{theorem}

\begin{proof}
    Fix some nondegenerate $\sigma$ and choose $\delta>0$ such that
    \begin{equation}\label{Eqn:DeltaCondition}
        \delta < \min \left\{\frac{D h_{N-1}}{(\sumH{N-1}+1)(\sumH{N-2}+1)}, \ D\left(\frac{1}{\sumSigH + 1} - \frac{1}{\sumH{N-1} + 1}\right) \right\} \ .
    \end{equation}
    Note that since $\sigma$ is nondegenerate, $\sumSigH < \sumH{N-1}$, so the second term in \eqref{Eqn:DeltaCondition} is strictly positive.

    Consider $x_0=D$ and the Huber function $f= \phi_{L,\eta}$ with
    \begin{equation*}
        \eta = \frac{D}{\sum_{i=0}^{N-1}h_i +1} +\delta \ .
    \end{equation*}
    From the first term in \eqref{Eqn:DeltaCondition}, we have $D - \eta (\sumH{N-2} + 1) > 0$ and therefore, $D-\eta\sumH{N-2} > \eta$. As a result, the first $N$ iterates of gradient descent are given by $x_k = D - \eta\sumH{k-1}$. Thus, any averaging $\sigma$ produces an output point
    \begin{align*}
        \xSpecial = \sum_{j=0}^{N} \sigma_j x_j & = \sum_{j=0}^N \sigma_j (D-\eta\sum_{i=0}^{j-1}h_i) = D-\eta\sumSigH \ .
    \end{align*}
    Then, applying the second term of \eqref{Eqn:DeltaCondition} with $\eta = \frac{D}{\sum_{i=0}^{N-1}h_i +1} +\delta$ yields
    \begin{align*}
        \xSpecial & = D- \left(\frac{D}{\sum_{i=0}^{N-1}h_i +1} +\delta\right)\sumSigH > \frac{D}{\sumH{N-1} +1} + \delta = \eta \ . 
    \end{align*}
    By construction, at $\xSpecial > \eta$ we have $\nabla f(\xSpecial) = \eta > \frac{D}{\sumH{N-1} +1}$. Since this bound holds for any nondegenerate averaging, we conclude the degenerate averaging $\sigma = (0,\dots,0,1)$ uniquely minimizes the worst-case gradient norm $\max_{(x_0,f)}\|\nabla f(\xSpecial)\|$.
\end{proof}

\subsubsection{Discussion of Conditions (\ref{Eqn:ConditionFVal}) and (\ref{Eqn:ConditionGradNorm})}
    We emphasize that \eqref{Eqn:ConditionFVal} and \eqref{Eqn:ConditionGradNorm} are essential to our proofs above; we rely on the fact that the Huber function is the worst-case problem instance for the terminal iterate, which is implied by \eqref{Eqn:ConditionFVal}/\eqref{Eqn:ConditionGradNorm}. Given any stepsize sequence $h$, we can numerically check if \eqref{Eqn:ConditionFVal}/\eqref{Eqn:ConditionGradNorm} hold by solving \eqref{Eqn:primal} and if averaging can improve the worst-case performance (see \eqref{Eqn:OptimizeSigma}).
    In \cref{Fig:C1C2Data}, we plot for $N=2$ and $N=3$ the stepsizes $h$ for which \eqref{Eqn:ConditionFVal} and \eqref{Eqn:ConditionGradNorm} hold numerically. Comparing with the values of $h$ for which averaging is not beneficial (which can also be numerically confirmed), we found that these conditions exactly align. That is, numerical results suggest that the converses of \cref{Thm:AvgFunc} and \cref{Thm:AvgGradNorm} are true, i.e., \eqref{Eqn:ConditionFVal}/\eqref{Eqn:ConditionGradNorm} hold if and only if averaging is not beneficial. We state this formally in the following conjecture.

    \begin{conjecture}\label{Conj:C1C2Converse}
        Consider gradient descent with stepsizes $h=(h_0, \dots, h_{N-1}) > 0$. For any $L,D > 0$, if the worst-case final iterate objective gap is \textbf{not} characterized by~\eqref{Eqn:ConditionFVal}, then there exists nondegenerate $\sigma$ such that
    \begin{equation*}
        \max_{(x_0,f)\in \mathfrak{F}_{L,D}} f\left(\sum^{N}_{j=0} \sigma_j x_j\right) - \inf f < \frac{LD^2}{4\sumH{N-1}+2} \ .
    \end{equation*}
    Similarly, if the worst-case final iterate gradient norm is \textbf{not} characterized by~\eqref{Eqn:ConditionGradNorm}, then there exists nondegenerate $\sigma$ such that 
    \begin{equation*}
         \max_{(x_0,f)\in \mathfrak{F}_{L,D}} \left\|\nabla f\left(\sum^{N}_{j=0} \sigma_j x_j\right) \right\| < \frac{LD}{\sumH{N-1}+1} \ .
    \end{equation*}
    \end{conjecture}

\begin{figure}
    \centering
    \includegraphics[width=1.0\textwidth]{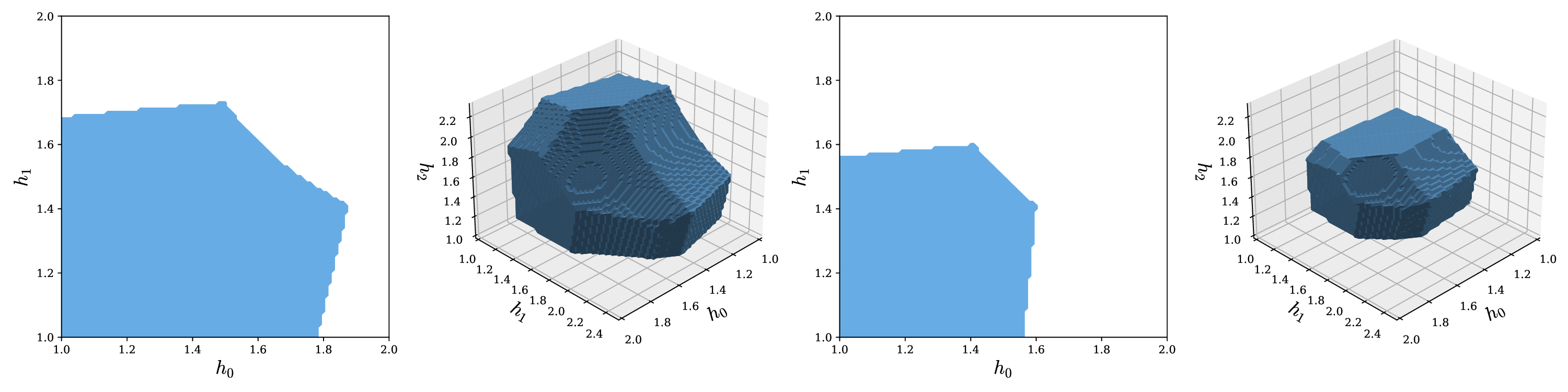}
    \caption{Values of $h$ for which \eqref{Eqn:ConditionFVal} holds (left) and \eqref{Eqn:ConditionGradNorm} holds (right) for $N=2,3$. Numerically we find that this region fully aligns with the stepsizes for which averaging is not beneficial (See \cref{Conj:C1C2Converse}).}
    \label{Fig:C1C2Data}
\end{figure}

\subsection{Extension to Projected/Proximal Gradient Descent}
    Reasoning virtually identical to the above can be applied to the proximal gradient method. We focus on the constant stepsize case, but the argument can be similarly extended to appropriate non-constant sequences. Consider minimizing an additive composite function $\psi = f +r $ where $f,r$ are both convex. Denote the proximal mapping by 
    \begin{equation*}
        \prox_{t r}(x) := \argmin \left\{r(u) + \frac{1}{2t}\|u-x\|^2: u \in \R^\dimension \right\}
    \end{equation*}
    and then iterate according to $x_{k+1}=\mathrm{prox}_{(h/L) r} (x_k-\frac{h}{L}\nabla f(x_k))$. We note that projected gradient descent is a special case where we take $r$ as the indicator function of a closed, convex set. In \cite[Theorem 7]{Teboulle_4T_Analytical} Teboulle and Vaisbourd proved for any $L$-smooth $f$, the proximal gradient method with constant stepsize $h \in (0,1]$ satisfies
    \begin{equation*}
        \psi(x_N) - \psi(\xStar) \leq \frac{LD^2}{4Nh}
    \end{equation*}
    and
    \begin{equation*}
       \|\psi_{h/L}(x_{N-1}) \| \leq \frac{LD}{Nh}
    \end{equation*}
    where $$\psi_{h/L}(x_{N-1}) = \left((x_{N-1} - \frac{h}{L}\nabla f(x_{N-1})) - (x_{N-1}^+ -\frac{h}{L} \nabla f(x_{N-1}^+))\right)L/h$$ and $x^+ = \mathrm{prox}_{(h/L)r} \left(x - \frac{h}{L}\nabla f(x) \right)$. Note that $\psi_{h/L}(x_{N-1}) \in \partial \psi(x_N)$, so these results are clear analogs of those in \cref{Thm:AvgFunc,Thm:AvgGradNorm}. Moreover, the bounds above are shown to be tight by the simple example of $f(x) = L \eta x$ and $r$ as the indicator function over $\R_+$, choosing $\eta = \frac{D}{2Nh}$ and $\eta = \frac{D}{Nh}$, respectively. The earliest occurrences of these tight lower bound examples to our knowledge are~\cite[Appendix 2.8]{Drori2014} and~\cite[Section 4]{Interpolation}.
    Consequently, we can apply similar arguments from \cref{Thm:AvgFunc} and \cref{Thm:AvgGradNorm}, to obtain the result below. As a note, although PEP is not used for proximal settings here, such natural extensions exist~\cite{Interpolation2} and have seen much use~\cite{Barre2023,pepit2022,Gu2020,Kim2021}.

    \begin{proposition}
    Consider the proximal gradient method with constant stepsize $h \in (0,1]$ and any $L,D>0$. For any nondegenerate averaging $\sigma$, the worst-case performance is strictly worse than that of the terminal iterate. That is,
        \begin{align*}
            \max_{(x_0,f)\in \mathfrak{F}_{L,D}} \psi(\xSpecial) - \psi(\xStar) > \frac{LD^2}{4Nh}
        \end{align*}
        and
        \begin{align*}
            \max_{(x_0,f)\in \mathfrak{F}_{L,D}} \|\xi\| > \frac{LD}{Nh}
        \end{align*}
        for any $\xi \in \partial \psi(\xSpecial)$.
    \end{proposition}

    \begin{proof}
        Let $x_0 = D$, $f(x) = L \eta x$, and let $r$ be the indicator function over $\R_+$.
        To prove our first claim, we let $\eta = \frac{LD}{2Nh}$. Our proximal gradient iterates are then given by $x_{k+1} = x_k - \frac{h}{L} \eta$. We observe that $\psi$ is linear over the convex hull of our iterates $x_0, \dots, x_N$. For any nondegenerate $\sigma$, we then have
        \begin{equation*}
            \psi(\xSpecial) - \psi(\xStar) = \psi(\xSpecial) = \sum_{j=0}^N \sigma_j f(x_j) > f(x_N) = f(D - hN\eta) = \frac{LD^2}{4Nh} \ .
        \end{equation*}

        To prove our second claim, fix some nondegenerate $\sigma$ and choose $\delta>0$ such that 
        \begin{equation} \label{Eqn:DeltaProxCondition}
            \delta <  \min \left\{\frac{D}{N(N-1)h}, \ D \left(\frac{1}{h\sum_{j=0}^N j \sigma_j} - \frac{1}{Nh}\right) \right\}
        \end{equation}
        Let $\eta = \frac{D}{Nh} + \delta$. By the first term of \eqref{Eqn:DeltaProxCondition}, we have $D - h(N-1)\eta > 0$ and therefore $x_k = D - \eta kh$ for $k = 1,\dots, N$.
        We similarly apply the second term of \eqref{Eqn:DeltaProxCondition} to get
        \begin{align*}
            \xSpecial & = \sum_{j=0}^N \sigma_j (D - \eta jh) =  D- \left(\frac{D}{Nh} +\delta\right) h \sum_{j=0}^N j \sigma_j > 0 \ . 
        \end{align*}
        Finally, given that $\xSpecial > 0$, we compute $\partial \psi(x_\sigma) = \{ L\eta \}$, so our claim holds. 
    \end{proof}

    \section{Extrapolations are Strictly Better than the Last Iterate}\label{Sect:SimpleExtrap}
Given that averaging gradient descent's iterates cannot improve its convergence bounds, we now remove the restriction that $\sigma$ describes a convex combination, allowing it to be freely chosen from $\R^{N+1}$. As discussed in \cref{Sect:Prelim}, this allows $\xSpecial$ to be anything in $x_0 + \text{span} \{ \nabla f(x_0), \dots, \nabla f(x_{N-1}) \}$.

An initial problem to investigate in this larger parameter space is to determine the optimal choice of $\sigma$ for a fixed number of iterations $N$ and fixed step sequence $h$. 
Given $L,D>0$ and stepsizes $h$, this amounts to solving the following nonconvex minimization problem
\begin{equation} \label{Eqn:OptimizeSigma}
    \min_{\sigma\in\mathbb{R}^{N+1}} p(\sigma) \ .
\end{equation}
Note that expanding the definition of $p(\sigma)$ gives a nonconvex minimax problem. Das Gupta et al.~\cite{BnB} introduced a spatial branch-and-bound approach tailored to numerically globally solving such problems. We apply their method to the new setting of optimizing extrapolation weights $\sigma$.
Fixing $N=10$ and constant stepsize $h=1$, \cref{Fig:SigmaGlobal} shows the optimal $\sigma$ for minimizing the objective gap $f(\xSpecial)-f(\xStar)$ and for minimizing the gradient norm $\|\nabla f(\xSpecial)\|$. We observe that the components of $\sigma$ are dominated by the weights of the last two components. It is worth noting, however, that approximating these final two weights and setting the remaining weights to zero significantly worsens the performance. For example, considering the objective gap above, one can approximate the optimal $\sigma$ by $(0,\dots, 0, -2.3493, 3.3493)$. Solving \eqref{Eqn:primal} for this approximated $\sigma$ gives a performance guarantee of $0.02473$, notably worse than the optimal value of $0.02000$, and worse than the guarantee without extrapolation of $0.02381$. Thus, the subtle structure (highlighted in \cref{Fig:SigmaGlobal}) for components $k < N-1$ appears to be integral to the optimal extrapolation scheme's success. The first four columns of \cref{tab:extrapolation-performance} show these best possible improvements to the worst-case performance as $N$ varies.

\begin{figure}[tbp]
\centering
\includegraphics[width=0.9\textwidth]{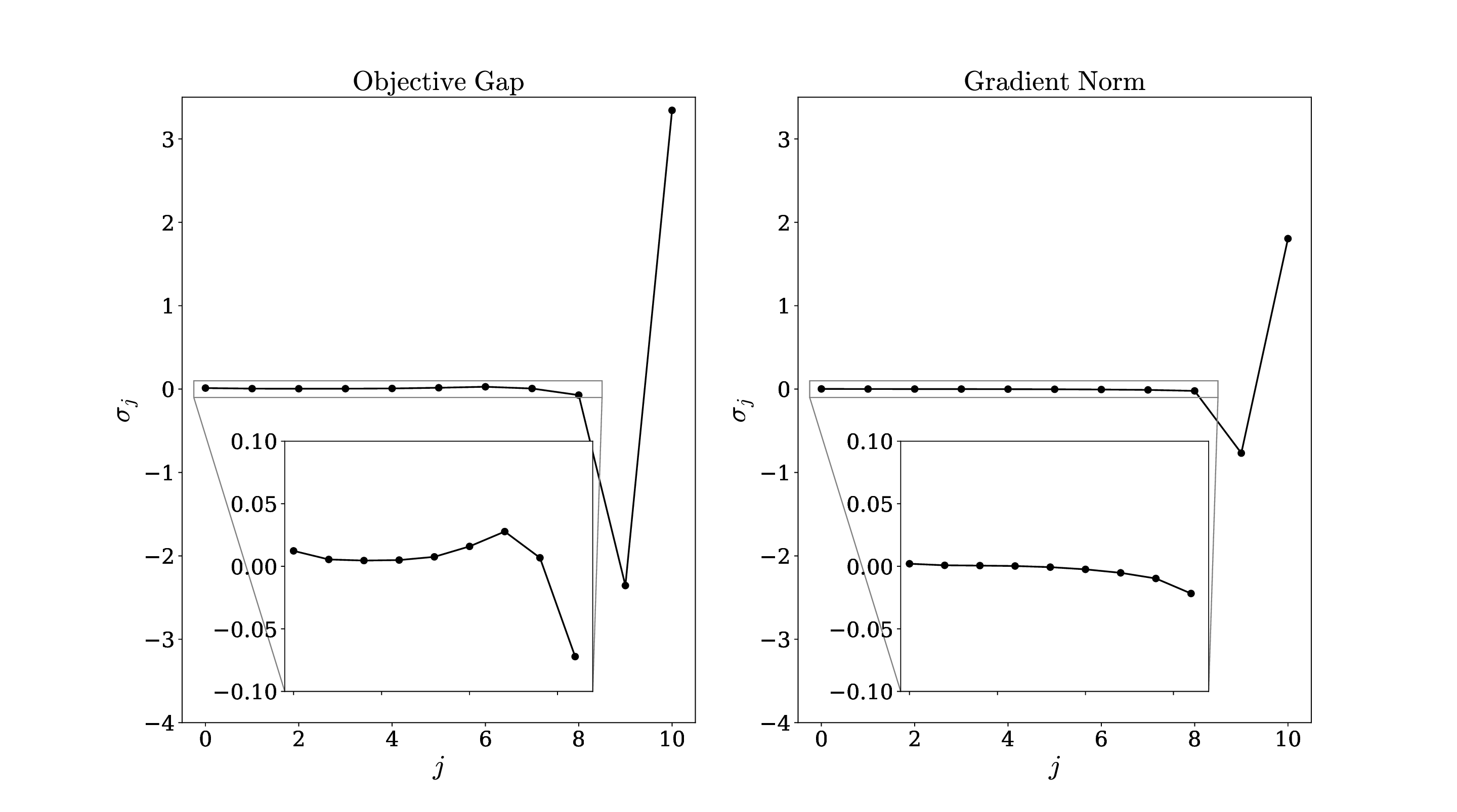}
\caption{Optimal extrapolation choice of $\sigma$ for $N=10$ under different performance measures. Under these respective choices of $\sigma$, the performance guarantees for $L=D=1$ are $0.02000$ for objective gap and $0.08443$ for gradient norm.}
\label{Fig:SigmaGlobal}
\end{figure}

\begin{table}[t]
\centering\footnotesize
\begin{tabular}{| c | c c | c c | c c |} 
 \hline
 & \multicolumn{2}{|c|}{Final Iterate} & \multicolumn{2}{|c|}{Optimal Extrapolation} & \multicolumn{2}{|c|}{Optimal Simple Extrapolation}  \\ [0.5ex] 
 \hline
 $N$ & Obj Gap & Grad Norm & Obj Gap & Grad Norm & Obj Gap & Grad Norm \\
 \hline
 5  & 0.0455 & 0.1667 & 0.0365 & 0.1483 & 0.0390 & 0.1546 \\ 
 10 & 0.0238 & 0.0909 & 0.0200 & 0.0844 & 0.0213 & 0.0871 \\
 15 & 0.0161 & 0.0625 & 0.0139 & 0.0592 & 0.0147 & 0.0607 \\ 
 20 & 0.0122 & 0.0476 & 0.0107 & 0.0456 & 0.0113 & 0.0466 \\
 25 & 0.0098 & 0.0385 & 0.0087 & 0.0371 & 0.0091 & 0.0378 \\
 \hline
\end{tabular}
\caption{Numerical comparison between gradient descent's worst-case performance for objective gap and gradient norm at its last iterate $x_N$, at the optimal extrapolation (the solution to \eqref{Eqn:OptimizeSigma}), and at the optimal simple extrapolation (the solution to \eqref{Eqn:OptimizeSigma} restricted to $\sigma = (-(c-1),0,
\dots,0,c)$). We fix $L=D=h=1$.}\label{tab:extrapolation-performance}
\end{table}

While the optimal $\sigma$ offers a noticeable improvement to the previous convergence bound, we did not identify a clear pattern in the optimal values. Even given a construction for the optimal extrapolation for each $N$, applying it would require prior knowledge of $N$ when beginning the method in order to construct $\xSpecial$ incrementally as a second vector in memory while running gradient descent. To avoid these shortcomings, we restrict $\sigma$ to a particularly convenient structure, namely $\sigma = (-(c-1),0,\dots,0,c)$ for any extrapolation factor $c \geq 1$. This corresponds to
\begin{equation}\label{Eqn:xSpecialSimple}
    \xSpecial = x_0 + c(x_N-x_0) = x_0 - c\sum_{k=0}^{N-1}\frac{h_k}{L} g_k \ .
\end{equation}
Hence, one can interpret $\xSpecial$ as extrapolating along the vector from $x_0$ to $x_N$. We refer to this particular method of extrapolation as \textit{simple extrapolation}. Note when $c=1$, this reverts back to simply returning gradient descent's last iterate $\xSpecial = x_N$. 

This scheme's simplicity offers a few advantages. To perform this extrapolation, one needs only the initial point $x_0$ and terminal point $x_N$; no auxiliary vectors or enlarged memory storage are required. Then, whenever the gradient descent iteration is stopped, this extrapolation can be computed, without requiring a priori knowledge of $N$. The constant $c$ can be selected on-the-fly once the method stops. See \cref{Cor:ExtrapSpecific} for several convenient, provably good formulas one could use, setting $c = 1 + \Theta(\sqrt{1/N\log(N)})$.

In the remainder of this section, we focus on gradient descent with constant stepsize $h \in (0,1]$ and demonstrate a provable benefit to simple extrapolation. As a result, we show that including this simple rescaling at the final step provides a provable improvement ``for free''. In \cref{tab:extrapolation-performance}, we include the optimal performance of simple extrapolation, though this will be discussed in more detail in \cref{SubSect:OptimalExtrapolation}.

\begin{remark}[Computational Considerations]\label{Remark:memoryCost}
    In large-scale optimization settings where memory is a \textit{severely} limiting factor, the simplicity of this extrapolation becomes relevant. While accelerated momentum methods achieve faster convergence, they require the simultaneous storage of two vectors rather than just the current iterate $x_k$. For such problem instances where memory usage is at capacity from storing a single vector $x_k$, doubling the memory storage is not an option, and hence momentum methods cannot be used. When $x_0$ is a structured point such as the origin, simple extrapolations can improve convergence while incurring no added memory costs.
    
    The increased memory costs of momentum methods are also avoided by the recently developed gradient descent accelerations from longer stepsizes in the sequence, e.g., the silver pattern's $O(1/N^{1.2716})$ rate~\cite{Silver_Accel}. Such methods may be state-of-the-art in certain severely memory-constrained settings. While our proven gains are restricted to the case $h \in (0,1]$, in \cref{Sect:Numerics}, we show numerically that simple extrapolation also (slightly) improves the performance for silver stepsizes in the worst case (and in several real-world examples), offering a computational benefit at nearly no cost. Thus, this post-processing method could improve the state-of-the-art method for memory-limited problems.
\end{remark}

\subsection{Improved Convergence Guarantees from Simple Extrapolations}
We are now prepared to present our main theorem quantifying the improvement in worst-case performance derived from simple extrapolations. To first discuss the consequences of our theory, we defer the proof of this result to \cref{SubSect:ProofOfExtrap}. Our analytic results below focus on the objective gap as the performance measure. In \cref{SubSect:OptimalExtrapolation}, we return to numerically consider the minimization of the reported gradient norm, establishing similar improvements. However, our \cref{Conj:GradNorm} suggests that simple extrapolation is less effective at improving the final gradient norm size, hence our focus on the objective gap.

\begin{theorem}\label{Thm:MainExtrapThm}
    For any convex $L$-smooth function $f$ with minimizer $\xStar$, consider gradient descent with constant stepsizes $h_k = h\in(0,1]$. For any $N > 0$, define the function
    \begin{equation}\label{Eqn:psi}
        \psi_N(c) := 1 - \sum_{i=0}^{N-1} \frac{\psiNum{c}}{\detDenom{c}} > 0
    \end{equation}
    and let $\cCrit$ be the largest root of $\psi_N(c)$. For any $c \in [1,\cCrit]$, $\xSpecial = x_0 + c(x_N-x_0)$ satisfies
    \begin{equation}\label{Eqn:NewBound}
        f(\xSpecial) - f(\xStar) \leq \frac{LD^2}{4Nhc+2}
    \end{equation}
    where $D=\|x_0-\xStar\|$.
    Moreover, for any $L,D > 0$, there exists a function $f$ where equality holds.
\end{theorem}

The choice of $\psi_N$ above will be later justified in \cref{SubSubSec:Step2}. In terms of PEP introduced in \cref{Sect:Prelim}, this theorem provides an exact description of the worst-case performance of small extrapolations: the extrapolation $\sigma = (-(c-1),0,\dots, 0, c)$ has
$$ p(\sigma) = \frac{LD^2}{4Nhc+2} \qquad \forall c\in [1,\cCrit]\ . $$
Setting $c$ as large as possible (within the theorem's bounds) provides the strongest performance guarantee. To quantify the level of resulting benefit, we next bound how large the root $\cCrit$ is for a given value of $N$.

The definition of $\cCrit$ as the largest root of $\psi_N(c)$ unfortunately yields no simple closed-form expression. Considering the case of $N=2$, our function is $\psi_2(c)=\frac{-64c^3+112c^2-36c-1}{12c-1}$, so $\cCrit$ is the largest root of $-64c^3+112c^2-36c-1=0$, approximately $1.312285$.
In general, given any $N$, $\cCrit$ will be the largest root of a degree $N+1$ polynomial. In \cref{Tbl:cStarData}, we include numerical approximations of $\cCrit$ for various values of $N$. The following proposition shows $\cCrit$ shrinks at rate $1 + \Theta\left(\frac{1}{\sqrt{N \log N}}\right)$.
\begin{proposition}\label{Prop:cCritBd}
    For $N > 1$, the largest root $\cCrit$ of $\psi_N(c)$ is bounded below by
    \begin{equation}
        \cCrit > \underbrace{1 -\frac{1}{4N} + \sqrt{\frac{1}{16N^2}+\frac{1}{(2N+1)(\harmonicUpperBd)}}}_{c_\ell} > 1 + \frac{1}{4\sqrt{N \log N}}
    \end{equation}
    where $\gamma\approx 0.5772$ is the Euler-Mascheroni constant. Moreover, this bound is nearly tight  as 
    \begin{equation*}
        \cCrit < \underbrace{1 -\frac{1}{4N} + \sqrt{\frac{1}{16N^2}+\frac{3}{(3N+1)(\harmonicLowerBd)}}}_{c_u} \ . 
    \end{equation*}
\end{proposition}

\begin{table}
    \centering\footnotesize
    \begin{tabular}{|c|c|c|c|c|}
        \hline
         $N$ & $\cCrit$ & $c_u$ & $c_\ell$ & $1+1/(4\sqrt{N \log N})$\\
         \hline
         $1$    & $1.5$      & $1.672029$ & $1.359869$ & - \\ 
         $10$   & $1.121974$ & $1.158494$ & $1.104917$ & $1.052099$ \\
         $100$  & $1.034804$ & $1.041404$ & $1.028570$ & $1.011650$ \\
         $10^4$ & $1.002878$ & $1.003171$ & $1.002235$ & $1.000824$ \\
         $10^6$ & $1.000246$ & $1.000263$ & $1.000186$ & $1.000067$ \\
         $10^8$ & $1.000022$ & $1.000023$ & $1.000016$ & $1.000006$ \\
        \hline
    \end{tabular}
    \caption{Approximate values of $\cCrit$ and the upper and lower bounds from \cref{Prop:cCritBd}.}
    \label{Tbl:cStarData}
\end{table}

This proposition is easily proven by establishing upper and lower bounds on $\psi_N(c)$ but involves some in-depth calculations. We therefore defer the proof to \cref{App:ProofOfCBound}.
Combining these simpler lower bounds on $\cCrit$ with \cref{Thm:MainExtrapThm} yields the following explicit convergence bounds for gradient descent with simple extrapolations.
\begin{corollary}\label{Cor:ExtrapSpecific}
     For any convex $L$-smooth function $f$ with minimizer $\xStar$ and $N > 1$, gradient descent with constant stepsizes $h_k=h\in(0,1]$ and simple extrapolation by factor
    \begin{equation}\label{Eqn:cLowerBd}
        c = 1 -\frac{1}{4N} + \sqrt{\frac{1}{16N^2}+\frac{1}{(2N+1)(\harmonicUpperBd)}}
    \end{equation}
    has $\xSpecial = x_0 + c(x_N-x_0)$ satisfy
    \begin{equation}\label{Eqn:ExtrapConvergenceSpec}
        f(\xSpecial) - f(\xStar) \leq \frac{LD^2}{4Nh+\left(\sqrt{1+\frac{16N^2}{(2N+1)(\harmonicUpperBd)}}-1\right)h  + 2 }
    \end{equation}
    where $D=\|x_0-\xStar\|$ and $\gamma\approx 0.5772$ is the Euler-Mascheroni constant. As a simpler, weaker bound, setting $c = 1 + \frac{1}{4 \sqrt{ N \log N}}$ ensures that
    \begin{equation}\label{Eqn:ExtrapConvergenceSimple}
        f(\xSpecial) - f(\xStar) \leq \frac{LD^2}{4Nh + \sqrt{\frac{N}{\log N}}h + 2} \ .  
    \end{equation}
\end{corollary}

This result shows that adding a simple extrapolation achieves the same worst-case convergence improvement as taking an additional $O(\sqrt{N/\log N})$ gradient steps. Note this gain increases with $N$ while the cost of the extrapolation step remains fixed (a vector-scalar multiplication and a vector sum). While this is still a lower-order improvement, this level of benefit from such a simple modification is rather surprising.

Finally, given the benefits of simple extrapolation proven above, we next consider its limitations. In the theorem below, we formalize the notion that too large of an extrapolation has a negative effect. An extrapolation that is too large may ``overshoot'' the minimizer, leading to worse performance. In particular, we show simple extrapolation factors any larger than $1+O(1/\sqrt{N})$ can perform strictly worse than the known worst case for gradient descent's last iterate.

\begin{proposition} \label{Prop:cHat}
    There exists a convex $L$-smooth function $f$ with minimizer $\xStar$ such that gradient descent with constant stepsizes $h_k = h\in(0,1]$ and any
    $$c > \cHat := \frac{1 + \sqrt{\frac{1}{2Nh+1}}}{1 - (1-h)^N} $$
    has $\xSpecial = x_0 + c(x_N-x_0)$ satisfy
    \begin{equation*}
        f(\xSpecial) - f(\xStar) > \frac{LD^2}{4Nh+2} \geq f(x_N) - f(\xStar) \ . 
    \end{equation*}
\end{proposition}
\begin{proof}
    Consider $f(x) = \frac{L}{2}x^2$ and $x_0=D$. Then gradient descent~\eqref{Eqn:GDStep} has iterates contract towards zero with $x_N = (1-h)^N D$. As a result, $\xSpecial = x_0 + c(x_N-x_0) = D + c(D(1-h)^N - D)$ and so $f(\xSpecial) = \frac{LD^2}{2}(1 + c((1-h)^N - 1))^2$. Plugging in the assumed lower bound on $c$, it immediately follows that
    $$f(\xSpecial) > \frac{LD^2}{2}\left(1 + \frac{1 + \sqrt{\frac{1}{2Nh+1}}}{1 - (1-h)^N}((1-h)^N - 1)\right)^2 = \frac{LD^2}{4Nh+2} \ . $$
    The bound on $f(x_N) - f(x_\star)$ follows from~\eqref{Eqn:ObjGapRate} (via \cite{FirstPEP}).
\end{proof}

From the results of \cref{Thm:MainExtrapThm} and \cref{Prop:cHat}, it is clear that there exists some optimal extrapolation factor $\cOpt \in [\cCrit, \cHat]$. Hence the optimal extrapolation factor shrinks at a rate between $1+\Theta(1/\sqrt{N\log(N)})$ and $1+\Theta(1/\sqrt{N})$.
In \cref{Sect:Numerics}, our results suggest that $\cOpt$ closely follows the $1+\Theta(1/\sqrt{N\log(N)})$ behavior of $\cCrit$, see \cref{Fig:cOptZoom}. Consequently, we expect the development of more sophisticated worst-case functions than the quadratic above would close this gap.

\subsection{Proof of Theorem~\ref{Thm:MainExtrapThm}} \label{SubSect:ProofOfExtrap}
Performance estimation provides the foundation for our exact analysis of the impact of simple extrapolation on worst-case guarantees. Recalling our previous notation for the associated PEP problem~\eqref{Eqn:Interpolation}, now parameterized by $c$ instead of $\sigma$, we denote
\begin{equation*}
    p(c) := \begin{cases} \max_{x_0,\xStar,f} \quad & f(\xSpecial) - f(\xStar) \\
                    \text{s.t.} \quad & f_i\geq f_j+\langle g_j,x_i-x_j\rangle +\frac{1}{2L}\| g_i - g_j\|^2 \quad \quad \forall i\neq j \in \mathcal{I}\\
                    & \|x_0-\xStar\| \leq D \\
                    & g_\star = 0 \\
                    & x_{k+1} = x_k - \frac{h}{L} g_k \quad \quad k=0,\dots,N-2\\
                    & \xSpecial = x_0 - \sum_{k=0}^{N-1} \frac{ch}{L} g_k \ . 
                \end{cases}
\end{equation*}
Then following the same SDP relaxation~\eqref{Eqn:primal} and applying weak duality, let $d(c)$ denote the upper bounding SDP~\eqref{Eqn:dual}, also parameterized by $c$ instead of $\sigma$:
\begin{equation}
    d(c) := \begin{cases}
        \min_{v,\lambda,Z} \quad  & v D^2 \\
        \text{s.t. } & \sum_{i\neq j} \lambda_{i,j} a_{i,j} - a_{\star,\idxSpecial} = 0 \\
        & v B_{0,\star} + \sum_{i\neq j} \lambda_{i,j} (A_{i,j}(h) + \frac{1}{2L}C_{i,j}) = Z  \\
        & Z \succeq 0 \\
        & v \geq 0, \lambda_{i,j} \geq 0  \quad \forall i\neq j \in \mathcal{I} \ .
    \end{cases}
\end{equation}
To prove our claimed upper bound on $f(\xSpecial)-f(\xStar)$, it then suffices to show the same upper bound on $d(c)$. This can be done by providing any dual feasible point with an objective value matching our claimed convergence guarantee.

In the special case of $c=1$ (that is, just reporting the last iterate without extrapolation),~\cite{FirstPEP} proved a bound matching our claim precisely through this process of constructing a dual certificate. Our derivation of a dual certificate is a generalization of their original method to simple extrapolations. In particular, our proof follows the same structure: we design a certificate with a slightly generalized form to accommodate $c$ and verify its feasibility via the same usage of Sylvester's criterion. The needed technical contribution lies in modifying their original certificate to account for $c$, identifying the function $\psi_N(c)$ (See \eqref{Eqn:psi}) exactly characterizing the feasibility of our resulting certificate, and quantifying its improved convergence rate.

Our construction of candidate dual solutions utilizes the following parameters
\begin{align}\label{Eqn:certComponents} 
        r_i &= \frac{ic}{2Nc-i+1} \quad \quad i=1,\dots,N\\
        t & = \frac{L}{2Nhc+1} \nonumber \ .
\end{align}
Given these, we consider dual solutions to $d(c)$ given by
\begin{align} \label{Eqn:dualVarVals}
    v &= \frac{1}{2}t \\
    \lambda_{i,j} &= \begin{cases}
        r_1 \quad & \text{if } i=-1,j=0 \\
        r_{j+1}-r_j \quad & \text{if } i=-1, 1\leq j \leq N-1 \\
        1-r_N \quad & \text{if } i=-1,j=N \\
        r_j \quad & \text{if } i=j-1, 1\leq j \leq N \\
        0 \quad & \text{otherwise}
    \end{cases}
\end{align}
where for simplicity we let indices $-1$ and $N$ correspond to $\star$ and $\idxSpecial$, respectively. When $c=1$ and the considered extrapolation disappears, this construction exactly reduces to that of \cite{FirstPEP}, in which the authors use $r_i = \frac{i}{2N-i+1}$ and $t = \frac{L}{2Nh+1}$ to prove the bound $LD^2/(4Nh+2)$.

We claim by construction these satisfy the first linear constraint $\sum_{i\neq j} \lambda_{i,j} a_{i,j} - a_{\star,\idxSpecial} = 0$.
Further, we claim that the second linear constraint is satisfied by setting $Z = \frac{1}{2}S_c(r,t)$ where $S_c(r,t)$ is defined as the following block matrix
\begin{equation}\label{Eqn:SMatrix}
    S_c(r,t) = \begin{pmatrix}
        t & q_c^T \\
        q_c & \quad \frac{1-h}{L}Q_c + \frac{h}{L}W_c
    \end{pmatrix}
\end{equation}
with components given by $q_c = (-r_1, r_1-r_2, \dots, r_{N-1}-r_N, r_N-1)^T$,
\begin{equation}
    Q_c = \begin{pmatrix} \label{Eqn:Mat0Def}
        2r_1 & -r_1 \\
        -r_1 & 2r_2 & -r_2 \\
        & -r_2 & 2r_3 & -r_3 \\
        & & \ddots & \ddots & \ddots \\
        & & & -r_{N-1} & 2r_N & -r_N \\
        & & & & -r_N & 1
    \end{pmatrix} \ , 
\end{equation}
and
\begin{equation}\label{Eqn:Mat1Def}
    W_c = \begin{pmatrix}
        2r_1 & r_2-r_1 & \dots & r_N-r_{N-1} & c-r_N\\
        r_2-r_1 & 2 r_2 & \dots & r_N-r_{N-1} & c-r_N \\
        \vdots & & \ddots & & \vdots \\
        r_N-r_{N-1} & r_N-r_{N-1} & \dots & 2r_N & c-r_N \\
        c-r_N & c-r_N & \dots & c-r_N & 1
    \end{pmatrix} \ . 
\end{equation}
For completeness, these two claimed identities are also verified in \cref{App:ScIdentity}. To complete our proof, we verify nonnegativity of $\lambda$ (\cref{SubSubSec:Step1}) and positive semidefiniteness of $Z$ (\cref{SubSubSec:Step2}) establishing dual feasibility, from which we can apply weak duality to prove our claimed rate (\cref{SubSubSec:Step3}). Lastly, we apply a Huber function to show our resulting worst-case bound is the best possible (\cref{SubSubSec:Step4}).

\subsubsection{Verification of Nonnegativity of \texorpdfstring{$\lambda$}{Lambda}} \label{SubSubSec:Step1}
We can easily confirm that $\lambda_{i,j} \geq 0$ for all $i,j \in \mathcal{I}$. Clearly $r_i > 0$ for all $i = 1,\dots,N$, provided that $c > \frac{N-1}{2N}$. And similarly, for $c > \frac{N-1}{2N}$, we see that $r_{i+1}-r_i = \frac{c(2Nc+1)}{(2Nc-i)(2Nc-i+1)} > 0$. Lastly, we have $1-r_N = \frac{Nc-N+1}{2Nc-N+1} > 0$ for $c > \frac{N-1}{N}$. We therefore have that $\lambda \geq 0$ for $c > \frac{N-1}{N}$, and the remaining verification steps will hold for such $c$ as well. However, for the purposes of simple extrapolation, we only consider $c\geq 1$.

\subsubsection{Verification of Positive Semidefiniteness of \texorpdfstring{$Z$}{Z}} \label{SubSubSec:Step2}
Note it is equivalent to show $S_c$ is positive semidefinite since it is just a rescaling of $Z$. 
Our first lemma is a generalization of the proof of \cite[Lemma 3.3]{FirstPEP} that shows the matrix $W_c$ is positive semidefinite for $c=1$. By allowing for general $c$, we obtain the following more complex determinant formulas.
\begin{lemma} \label{Lem:detFormula}
    Let $W_c$ be defined as in \eqref{Eqn:Mat1Def} with entries set as in \eqref{Eqn:certComponents} and denote the $k$-th principal submatrix of $W_c$ by $M_k(c)$. Then for $k=0,\dots,N-1$,
    \begin{equation}\label{Eqn:detMk}
        \det M_k(c) = c^{k+1}\left(1+\sum_{i=0}^k \frac{2Nc-2k-1}{\detDenom{c}}\right) \prod_{i=0}^k \frac{\detDenom{c}}{(2Nc-i)^2}
    \end{equation}
    and
    \begin{equation}\label{Eqn:detMN}
        \det M_N(c) = c^N \left(1 - \sum_{i=0}^{N-1} \frac{\psiNum{c}}{\detDenom{c}}\right) \prod_{i=0}^{N-1} \frac{\detDenom{c}}{(2Nc - i)^2} \ .
    \end{equation}
\end{lemma}

The proof of this result is deferred to \cref{App:ProofOfDetFormula}. This explicit formula for the determinant yields the following result on the positive definiteness of $W_c$.
\begin{lemma}\label{Lem:Mat1PD}
    Let $W_c$ be defined as in \eqref{Eqn:Mat1Def} with entries set as in \eqref{Eqn:certComponents}. Then $W_c$ is positive definite if and only if $\psi_N(c) > 0$, with $\psi_N(c)$ defined as in \eqref{Eqn:psi}
\end{lemma}
\begin{proof}
    Recall from Sylvester's criterion that a symmetric matrix is positive definite if and only if all leading principal submatrices of a matrix have positive determinant. We can easily see from \eqref{Eqn:detMk} that for $k<N$, $\det M_k(c)$ is the product of all positive factors, so $\det M_k(c) > 0$. Therefore, by Sylvester's criterion, $W_c$ is positive definite if and only if $\det M_N(c)$. Finally, observe that in \eqref{Eqn:detMN}, $\det M_N(c)$ is the product of $\psi_N(c)$ with a sequence of positive factors. We therefore conclude that $\det M_N(c) > 0$ if and only if $\psi_N(c) > 0$ and the result follows.
\end{proof}

We now examine the matrix $Q_c$.
\begin{lemma}\label{Lem:Mat0PD}
    Let $Q_c$ be defined as in \eqref{Eqn:Mat0Def} with entries set as in \eqref{Eqn:certComponents}. Then $Q_c$ is positive definite.
\end{lemma}
\begin{proof}
    Consider any $y = (y_0,\dots,y_N)^T \neq 0$. We calculate
    \begin{align*}
        y^T Q_c y &= \sum_{i=0}^{N-1} 2 r_{i+1} y_i^2 - 2\sum_{i=0}^{N-1} r_{i+1} y_i y_{i+1} + y_N^2 \\
        & = \sum_{i=0}^{N-1} r_{i+1}(y_{i+1}-y_i)^2 + r_1 y_0^2 + \sum_{i=1}^{N-1} (r_{i+1}-r_i)y_i^2 + (1-r_N)y_N^2 \ .
    \end{align*}
    The positivity of the first and second terms is immediate, and it was shown in \cref{SubSubSec:Step1} that $r_{i+1}-r_i > 0$ and $1-r_N > 0$. Therefore, we have $y^T Q_c y>0$.
\end{proof}

Finally, we combine each of these results to prove the following claim
\begin{lemma}
    If $\psi_N(c)>0$ then $S_c(r,t)$ is positive semidefinite.
\end{lemma}
\begin{proof}
For simplicity, we will denote $S_c = S_c(r,t)$. Since $\psi_N(c)>0$, by \cref{Lem:Mat0PD,Lem:Mat1PD}, we know that both $Q_c$ and $W_c$ are positive definite. Observe that for $h\in(0,1]$, $\frac{1-h}{L}Q_c+\frac{h}{L}W_c$ must be positive definite as it is a convex combination of positive definite matrices. We can therefore consider the Schur complement of $\frac{1-h}{L}Q_c+\frac{h}{L}W_c$ in $S_c$. From a well-known result on the Schur complement, we know that $S_c$ is positive semidefinite if and only if $t - q_c^T\left( \frac{1-h}{L}Q_c + \frac{h}{L} W_c\right)^{-1} q_c \geq 0$. However, we also have the standard identity
    \begin{equation}\label{Eqn:SchurComp}
        \det S_c = \left(t - q_c^T\left( \frac{1-h}{L}Q_c + \frac{h}{L} W_c\right)^{-1} q_c\right) \det \left( \frac{1-h}{L}Q_c+\frac{h}{L} W_c\right)    
    \end{equation}
    Next define the vector $u_c = (\frac{2Nhc+1}{L},1,\dots,1)^T$. It is easily verified that $S_c u_c = 0$ (see \texttt{Mathematica proof 4.2}). Consequently, $\det S_c = 0$. But $\det \left( \frac{1-h}{L}Q_c+\frac{h}{L} W_c\right) > 0$, so from \eqref{Eqn:SchurComp} we must have 
    \begin{equation*}
        \left(t - q_c^T\left( \frac{1-h}{L}Q_c + \frac{h}{L} W_c\right)^{-1} q_c\right) = 0.
    \end{equation*}
    Therefore, $S_c$ is positive semidefinite.
\end{proof}
All that remains is to show selecting $c<\cCrit$ ensures $\psi_N(c)>0$. For $c>1$, the denominator $\detDenom{c}$ is always positive, so $\psi_N(c)$ is a continuous rational polynomial. And note that $\psi_N(1) = 1 > 0$. A quick calculation verifies that $\psi_N'(c)<0$ for all $c>1$ with $\lim_{c\rightarrow \infty}\psi_N(c) =-\infty$. From this, $\psi_N(c)$ has exactly one root larger than 1, which must be $\cCrit$. Moreover, by continuity, $\psi_N(c) \geq 0$ for $c \in [1,\cCrit]$ and $\psi_N(c) < 0$ for $c > \cCrit$.

\subsubsection{Deriving Claimed Extrapolation Guarantee from Weak Duality} \label{SubSubSec:Step3}
Since our proposed dual certificate is feasible, using that $v = \frac{1}{2}t = \frac{L}{4Nhc+2}$ as in \eqref{Eqn:dualVarVals}, we conclude that $d(c) \leq vD^2 = \frac{LD^2}{4Nhc+2}$. Then, by applying weak duality, for any $L$-smooth convex $f$,
\begin{equation*}
    f(\xSpecial) - f(\xStar) \leq p(c) \leq d(c) \leq \frac{LD^2}{4Nhc+2} \ .
\end{equation*}

\subsubsection{Claimed Simple Extrapolation Guarantee is Tight}  \label{SubSubSec:Step4}
Finally, we demonstrate that our bound \eqref{Eqn:NewBound} is tight. Our proof is an adjustment of the tight example in~\cite[Theorem 3.2]{FirstPEP} to account for the extrapolation $c$. Consider $x_0=D$ and $f=\phi_{L,\eta}$ with $\eta = D/(2Nhc+1)$.
Then gradient descent's iterates are $x_k = D(1-kh\eta)$ for $k=1,\dots,N$. Taking a simple extrapolation yields
    \begin{align*}
            \xSpecial &= x_0 + c(x_N-x_0) = D - Nhc\eta 
            = (Nhc+1)\eta \ . 
    \end{align*}
    Since $\xSpecial > \eta$, we can therefore evaluate
    \begin{align*}
        f(\xSpecial) -f(\xStar) &= f(\xSpecial) - 0= L\eta^2(Nhc+1) - \frac{L\eta^2}{2}  = \frac{LD^2}{4Nhc+2} \ . 
    \end{align*}

\subsubsection{Remarks on Assumptions of the Proof of Theorem~\ref{Thm:MainExtrapThm}}
We briefly draw attention to our claim in \cref{SubSubSec:Step2} above that $\psi_N(c) < 0$ for $c > \cCrit$. This fact shows that our particular certificate construction breaks down for $c>\cCrit$. Indeed, our numerical experiments in \cref{Sect:Numerics} indicate that \eqref{Eqn:NewBound} is, in fact, false for $c>\cCrit$.
Similarly, our certificate construction breaks down for constant stepsizes beyond the assumed $h \leq 1$. If $h>1$, then we can no longer guarantee that $\frac{1-h}{L}Q_c + \frac{h}{L}W_c$ is positive definite. A different family of dual certificates would be needed to provide guarantees under extrapolation with stepsizes longer than one.

    \section{Numerical Extensions} \label{Sect:Numerics}
Beyond enabling the preceding proofs, the associated performance estimation problems enable numerically surveying the effects of averaging and extrapolation. Here, we address three extensions of our theory, utilizing \texttt{Mosek}~\cite{mosek} via \texttt{JuMP}~\cite{JuMP} for any numerical solves. First, \cref{SubSect:OptimalExtrapolation} shows results paralleling those of \cref{Sect:SimpleExtrap} appear to hold when considering the worst-case gradient norm rather than the worst-case objective gap. For both of these settings, we compute the simple extrapolation factor minimizing the worst-case performance in both the objective gap and gradient norm, finding our theory nearly matches the optimal factor. Then, in contrast to the many prior works finding one-dimensional functions characterize the worst-case behavior of gradient descent's last iterate, \cref{SubSect:OneDimensional} computes its worst-case performance with and without a restriction to one-dimensional problems, showing these are not sufficient to describe the performance of extrapolations. Lastly, \cref{SubSect:NumericsOther} surveys the effect of simple extrapolations across a range of (accelerated) gradient methods, finding consistent, small benefits.

Recall the PEP problem~\eqref{Eqn:Interpolation} can be written equivalently as the SDP~\eqref{Eqn:primal} provided the dimension $\dimension$ of functions $f$ considered is at least $N+2$. Consequently, if no restriction is made on the dimension of the considered problems (i.e., high-dimension examples are allowed), we can numerically compute optimal solutions to the PEP problem via any commercial interior point method solver. Here, we consider both the worst case measured in terms of the objective gap at $\xSpecial$, as well as the gradient norm $\|\gSpecial\| = \| \nabla f(\xSpecial)\|$. To distinguish these, we denote the performance estimation problem maximizing the reported point's objective gap (as was considered throughout \cref{Sect:SimpleExtrap}) by
\begin{equation}\label{Eqn:pObj}
    p_{\text{obj}}(c) = \begin{cases} \max_{F,G} \quad & F a_{\star,\idxSpecial} \\
                    \text{s.t. } & Fa_{i,j}+\Tr G A_{i,j}(h) + \frac{1}{2L}\Tr G C_{i,j} \leq 0  \quad \quad \forall i\neq j \in \mathcal{I}\\
                    & G \succeq 0 \\
                    & \Tr G B_{0,\star} \leq D^2 \ . 
                \end{cases}
\end{equation}
We denote the alternative performance estimation problem maximizing the reported point's gradient norm, formulated as an SDP, by
\begin{equation}    \label{Eqn:pGrad}
    p^2_{\text{grad}}(c) = \begin{cases} \max_{F,G} \quad & \Tr G C_{\star,\idxSpecial} \\
                    \text{s.t. } & Fa_{i,j} + \Tr G A_{i,j}(h) + \frac{1}{2L}\Tr G C_{i,j} \leq 0 \quad \quad \forall i \neq j \in \mathcal{I}\\
                    & G \succeq 0 \\
                    & \Tr G B_{0,\star} \leq D^2  \ . 
                \end{cases}
\end{equation}
Note that due to the structure of $G$, our SDP in 
\eqref{Eqn:pGrad} solves for $\|g_\sigma\|^2$ rather than $\|g_\sigma\|$; we define $p_{\text{grad}}(c)$ to correspond to the non-squared value, $\|g_\sigma\|$.
In both cases, the equalities above hold by virtue of assuming $m\geq N+2$.

\subsection{Optimal Extrapolations for Objective Gap and Gradient Norm}\label{SubSect:OptimalExtrapolation}
Numerically, we find much of our results on the benefits of extrapolation on the worst-case final objective gap carry over to the worst-case final gradient norm.
However, the exact dual certificates $v,\lambda, Z$ underlying our proof do not constitute proofs for the setting of gradient norms. We expect dual certificates proving a benefit from small extrapolations exist but have not identified their structure and hence leave proving such parallel results open. Theory similar to the H-duality theory of~\cite{Hdual} may facilitate doing so without repeating and re-engineering the proof.

In \cref{Fig:cOptZoom}, we plot both $p_{\text{obj}}$ and $p_{\text{grad}}$ as the extrapolation factor $c$ varies for fixed $N=7$, $h=1$, and $L=D=1$. In overall structure, these two plots look very similar. For $c$ in an interval near one, \cref{Thm:MainExtrapThm} ensures $p_{\text{obj}}(c) = LD^2/(4Nhc+2)$. We see similarly in the gradient norm case that an interval near one has $p_{\text{grad}}(c) = LD/(Nhc+1)$, improving on its tight last iterate 
bound~\cite{Teboulle_4T_Analytical} of $\|\nabla f(x_N) \| \leq \frac{LD}{Nh+1}$. Reasoning mirroring the tightness of Huber functions in \cref{Thm:MainExtrapThm} provides an example attaining this conjectured gradient norm rate under extrapolation. Let $x_0 = D$ and again consider $f = \phi_{L,\eta}$ as in \eqref{Eqn:Huber} with $\eta = \frac{D}{Nhc+1}$. A simple calculation shows that $x_N = D-Nh\eta$ and 
\begin{equation*}
    \xSpecial = x_0 + c(x_N-x_0) = D - \frac{DNhc}{Nhc+1} = \frac{D}{Nhc+1} = \eta \ .
\end{equation*}
We then have $\|\gSpecial \| = \| \nabla f(\eta)\| = \frac{LD}{Nhc+1}$.

\begin{figure}
\centering
\includegraphics[width=1.0\textwidth]{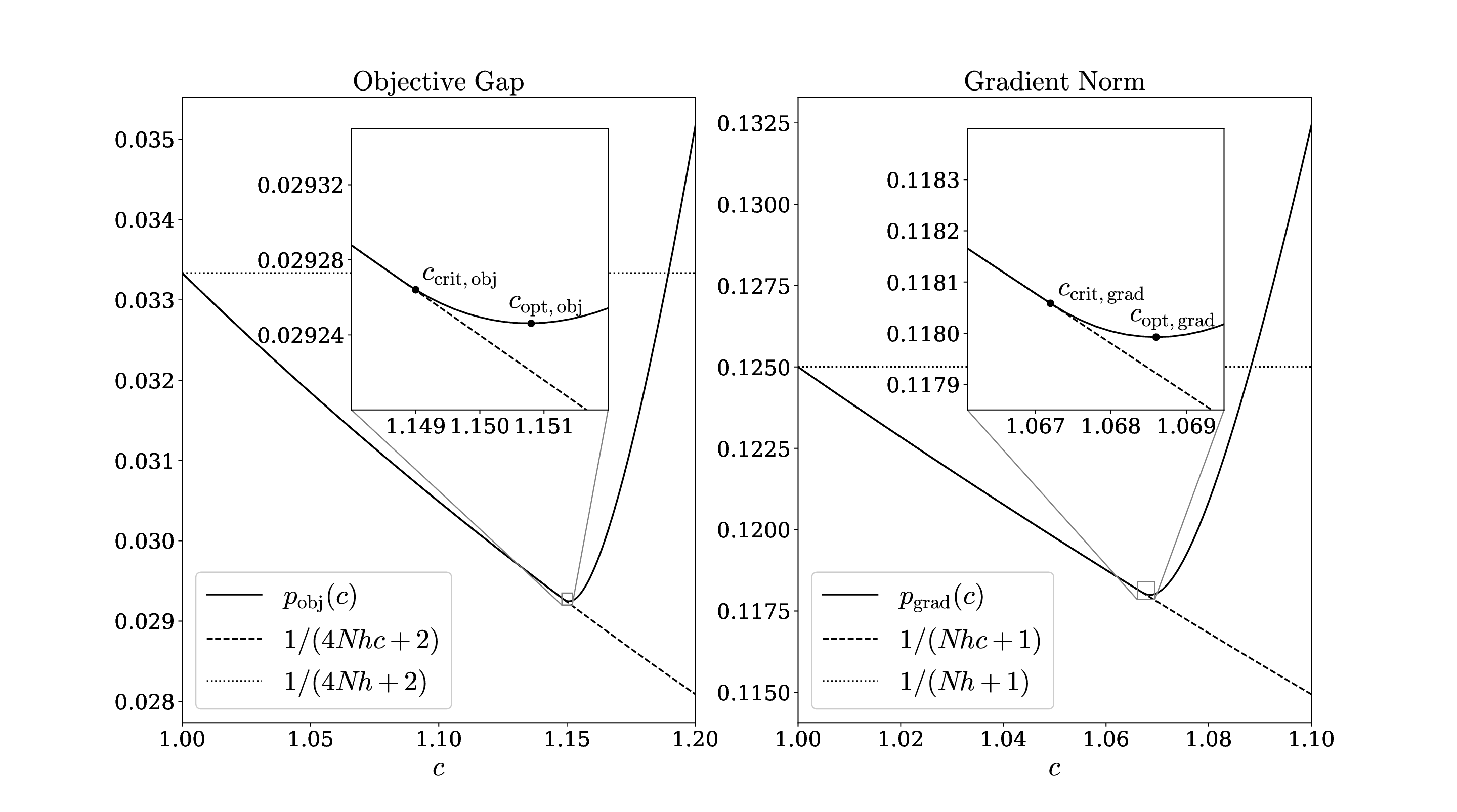}
\caption{Near-optimality of $\cCrit$ for objective gap and gradient norm with $N=7, L=D=h=1$.}
\label{Fig:cOptZoom}
\end{figure}

We consider two notable values of $c$ for each performance measure, the critical point $\cCrit$ where $p(c)$ deviates from the simple formulas $\frac{1}{4Nhc+2}$ and $\frac{1}{Nhc+1}$ (as discussed in \cref{Sect:SimpleExtrap}) and the point $\cOpt$ minimizing $p(c)$. We denote these by
\begin{align*}
    c_{\text{crit,obj}} &= \sup\left\{ c\geq 1 \mid p_{\text{obj}}(c) = \frac{LD^2}{4Nhc+2}\right\} \ , &
    c_{\text{opt,obj}} &= \mathrm{argmin\ } p_{\text{obj}}(c) \ , \\
    c_{\text{crit,grad}} &= \sup\left\{ c\geq 1 \mid p_{\text{grad}}(c) = \frac{LD}{Nhc+1}\right\} \ , &
    c_{\text{opt,grad}} &= \mathrm{argmin\ } p_{\text{grad}}(c) \ .
\end{align*}
For both performance measures in \cref{Fig:cOptZoom}, we see the critical point characterized by our theory is nearly but not quite optimal. Contrasting these measures, note that $\cCritGrad \neq c_{\text{crit,obj}}$, with $\cCritGrad$ consistently less than $c_{\text{crit,obj}}$, as shown in \cref{Tbl:cCritcOpt}. This indicates that extrapolations should be more conservative when seeking a small gradient norm. Moreover, numerically, we observe a faster rate of decay for $\cCritGrad$ as compared to $c_{\text{crit,obj}}$; this is clearly illustrated in \cref{Fig:cCritcOpt}. Specifically, we see that $\cCritGrad$ appears to be shrinking at rate $1+O(1/N)$, motivating the following conjecture.
\begin{conjecture}\label{Conj:GradNorm}
    Consider a gradient descent algorithm of fixed length $N$, constant stepsize $h \in (0,1]$, and with $\|x_0 - \xStar \| = D$. The critical value $\cCritGrad = 1 +\Theta(1/N)$, ensuring for any $c \in [1,\cCritGrad]$, simple extrapolation by factor $c$ has a tight worst-case bound of $\| \nabla f(\xSpecial)\| \leq \frac{LD}{Nhc+1}$.
\end{conjecture}

If this conjecture is correct, simple extrapolations up to the critical $\cCritGrad$ would be a provably less effective strategy for reducing the gradient norm than for the objective gap. Whereas an extrapolation of size $1+\Theta(1/\sqrt{N\log(N)})$ gave the same improvement in objective value as $\Theta(\sqrt{N/\log(N)})$ additional gradient steps, a simple extrapolation of size $1+\Theta(1/N)$ would only improve the gradient norm guarantee by the same amount as a constant number of additional gradient steps.

\begin{figure}[tbp]
    \centering
    \begin{minipage}[b]{0.51\textwidth}
        \centering
        {\renewcommand{\arraystretch}{1.7}    \footnotesize
        \begin{tabular}{|c||c|c||c|c|}
        \hline
             $N$ & $c_\text{crit,obj}$ & $c_\text{opt,obj}$ & $\cCritGrad$ & $\cOptGrad$ \\
             \hline
             1  & 1.5   &  1.5     & 1.4142 & 1.4142 \\
             5  & 1.1800 & 1.1821  & 1.0931 & 1.0950 \\
             10 & 1.1220 & 1.1238  & 1.0471 & 1.0481 \\
             25 & 1.0739 & 1.0752  & 1.0182 & 1.0187 \\
             50 & 1.0507 & 1.0516  & 1.0086 & 1.0090 \\
             \hline
        \end{tabular}}
        \captionsetup{type=table}
        \caption{The critical and optimal extrapolation factors for objective gap and gradient norm.}
        \label{Tbl:cCritcOpt}
    \end{minipage}
    \hfill
    \begin{minipage}[b]{0.47\textwidth}
        \vspace{0pt}
        \centering
        \includegraphics[width=\textwidth]{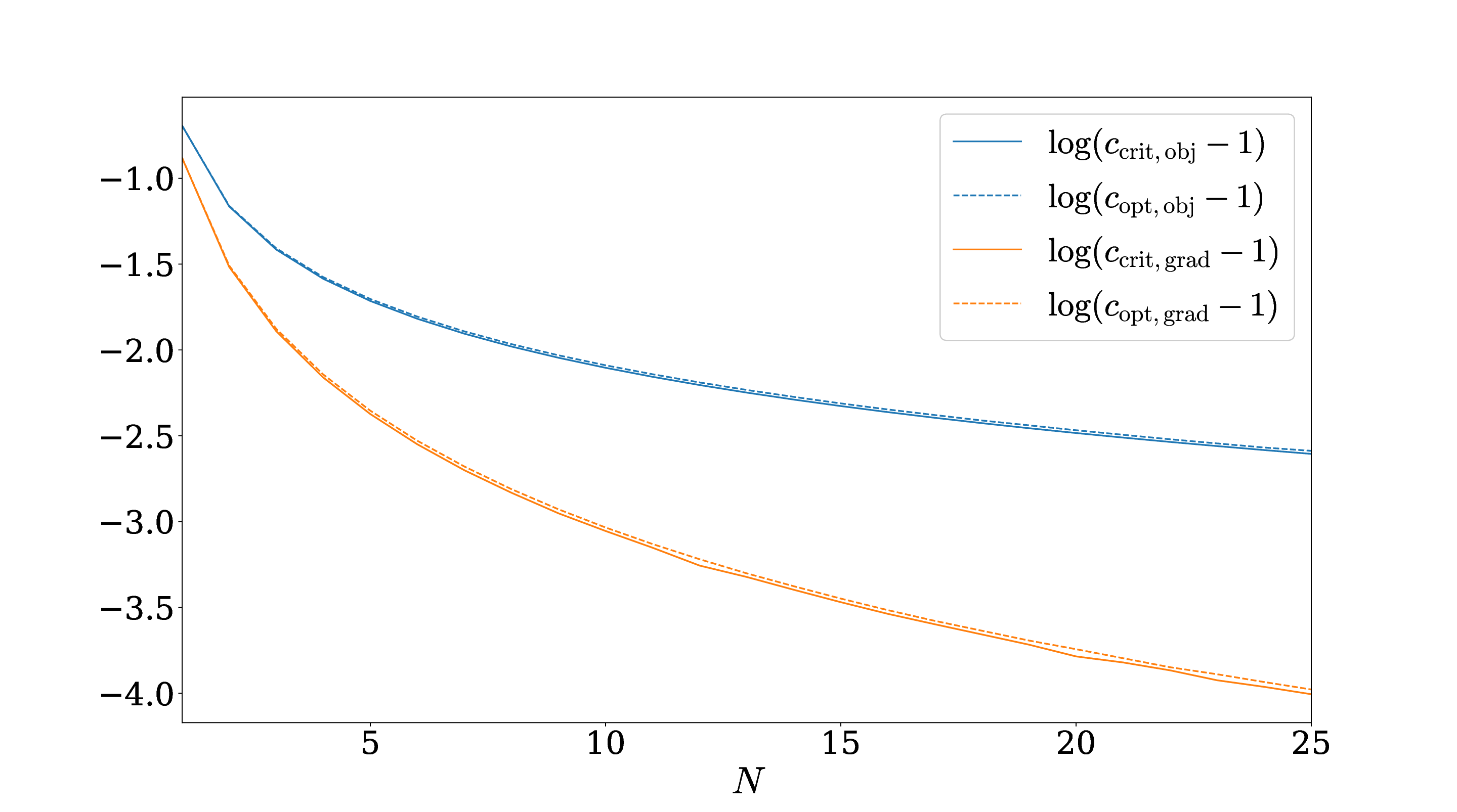}
        \caption{Similarity of critical and optimal factors and divergence between objective and gradient.}
        \label{Fig:cCritcOpt}
    \end{minipage}
\end{figure}

\subsection{Limitations of One-Dimensional Worst-Case Functions} \label{SubSect:OneDimensional}
For $c\leq \cCrit$, we have shown in \cref{Thm:MainExtrapThm} that there exists a one-dimensional function that attains our worst-case performance. Similarly, regarding averaging, in \cref{SubSect:AvgTheory}, we showed that one-dimensional functions were sufficient to describe the worst-case behavior of any averaging scheme $\sigma$. A natural question to ask in light of this is whether the same is true for simple extrapolation, that is, whether one-dimensional  functions are sufficient for attaining the worst-case behavior for all simple extrapolations. To answer this, we can modify our original SDP \eqref{Eqn:primal} by adding an additional constraint that $G = ww^T$ for some vector $w \in \R^{N+2}$. This yields the problem
\begin{equation} \label{Eqn:primal1D}
    p_{\text{obj,1D}}(c) := \begin{cases} \max_{F,w} \quad & F a_{\star,\idxSpecial} \\
                    \text{s.t. } & Fa_{i,j} + \Tr G A_{i,j}(h) + \frac{1}{2L}\Tr G C_{i,j} \leq 0, \quad\forall i \neq j \in \mathcal{I}\\
                    & G \succeq 0 \\
                    & \Tr G B_{0,\star} \leq D^2 \\
                    & G = ww^T \ .
                \end{cases}
\end{equation}

Due to the rank-one constraint, this problem becomes nonconvex, so we must now perform global rather than local optimization. However, for small $N$, this is still computationally feasible. In \cref{Fig:primal1DData}, we plot our results for $N=7$ with $h=1$. For $c < \cCrit$, we already know the worst-case function is $\phi_{L,\eta}$ from \cref{Thm:MainExtrapThm}. \cref{Fig:primal1DData} shows that for $c>\cCrit$, the worst-case performance, in general, deviates from the worst-case performance on one-dimensional objectives. Hence, characterizing the optimal factor $c_{\text{opt}}$ and measuring the (small) gap between it and $\cCrit$ will fundamentally require the design of worst-case problem instances beyond Huber functions and quadratics.

The dimension of problems required can be predicted by examining the rank of optimal primal solutions $G$ of~\eqref{Eqn:primal}. To see this, recall that the columns of the Cholesky decomposition, $G = H^T H$, gives the gradient vectors seen in that worst-case problem instance. Numerically, even for large $c$, we find the optimal $G$ is rank two, indicating the worst-case problem instances for $c>\cCrit$ only require one additional dimension. We show an example of one such worst-case function in \cref{Fig:Function2D}.

\begin{figure} 
    \centering
    \includegraphics[width=0.4\textwidth]{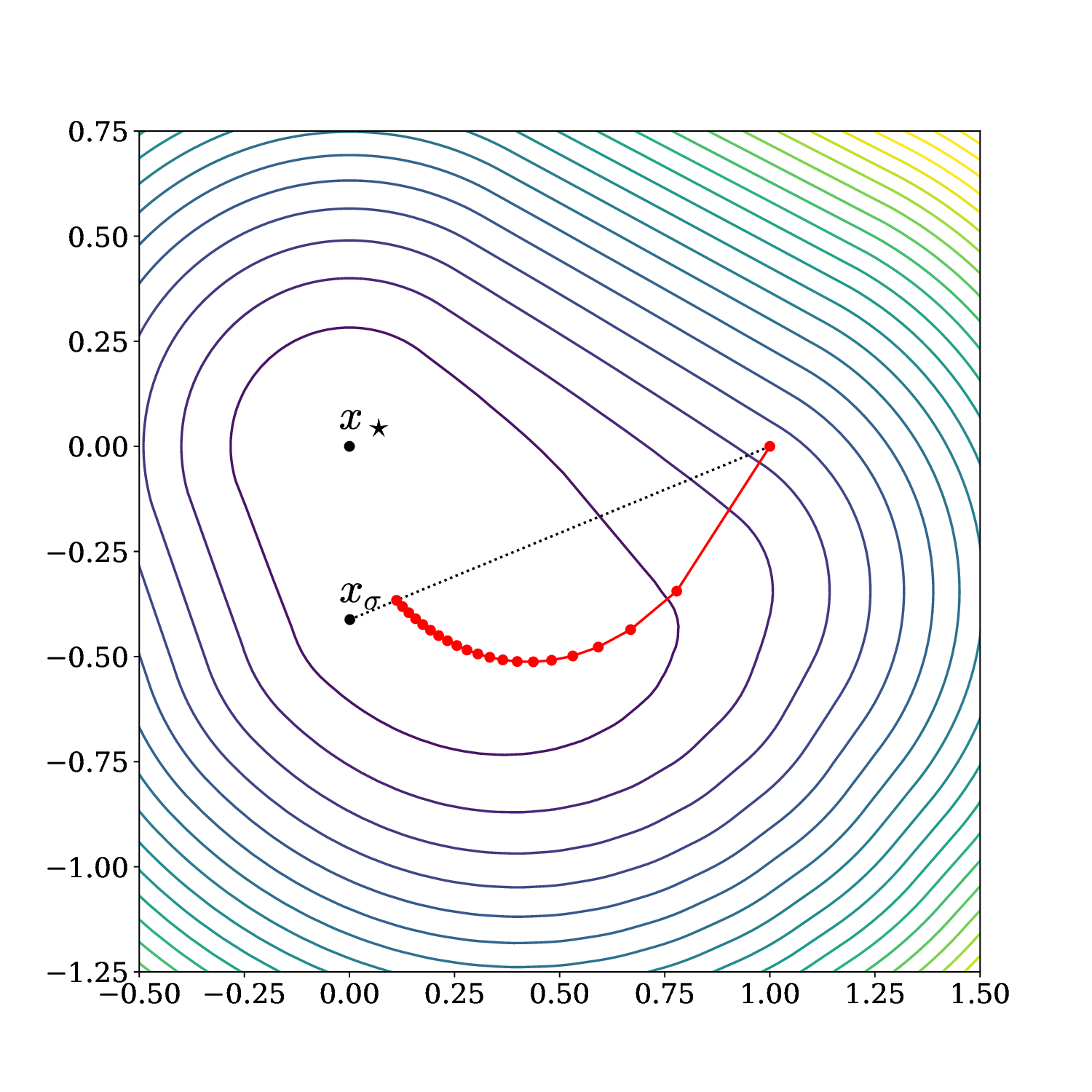}
    \caption{Level sets of a two-dimensional worst-case function for simple extrapolation with $N=20$ and $c = 1.125$.}
    \label{Fig:Function2D}
\end{figure}

Our numerical results indicate that among one-dimensional functions, the worst case is always a Huber function or a quadratic function. This can be seen in \cref{Fig:primal1DData} as $p_{\text{obj,1D}}$ appears to be piecewise smooth with two distinct pieces. The first piece, as discussed above, numerically equals $LD^2/(4Nhc+2)$, which is attained by considering $\phi_{L,\eta}$. The second piece numerically matches the curve $\frac{LD^2}{2}(1-c+c(1-h)^N)^2$; this is the exact value achieved by simple extrapolation on the quadratic function $f(x) = \frac{L}{2}x^2$. The analogous result for gradient norm seems to hold numerically as well. Therefore, in one dimension, the worst-case instances of simple extrapolation seem to be easily described. Setting $h=1$ and computing the minimum of these two functions, the best extrapolation factor in one dimension is
$ c=1-\frac{1}{4N} + \sqrt{\frac{1}{2N} + \frac{1}{16N^2}} $
which attains a convergence rate of $LD^2/(4N+\sqrt{8N+1}+1)$. This provides a lower bound on the magnitude of rate improvement that simple extrapolation that can achieve among problems of any dimension, which is within a log term of our \cref{Cor:ExtrapSpecific}.

\begin{figure}
\centering
\includegraphics[width=1.0\textwidth]{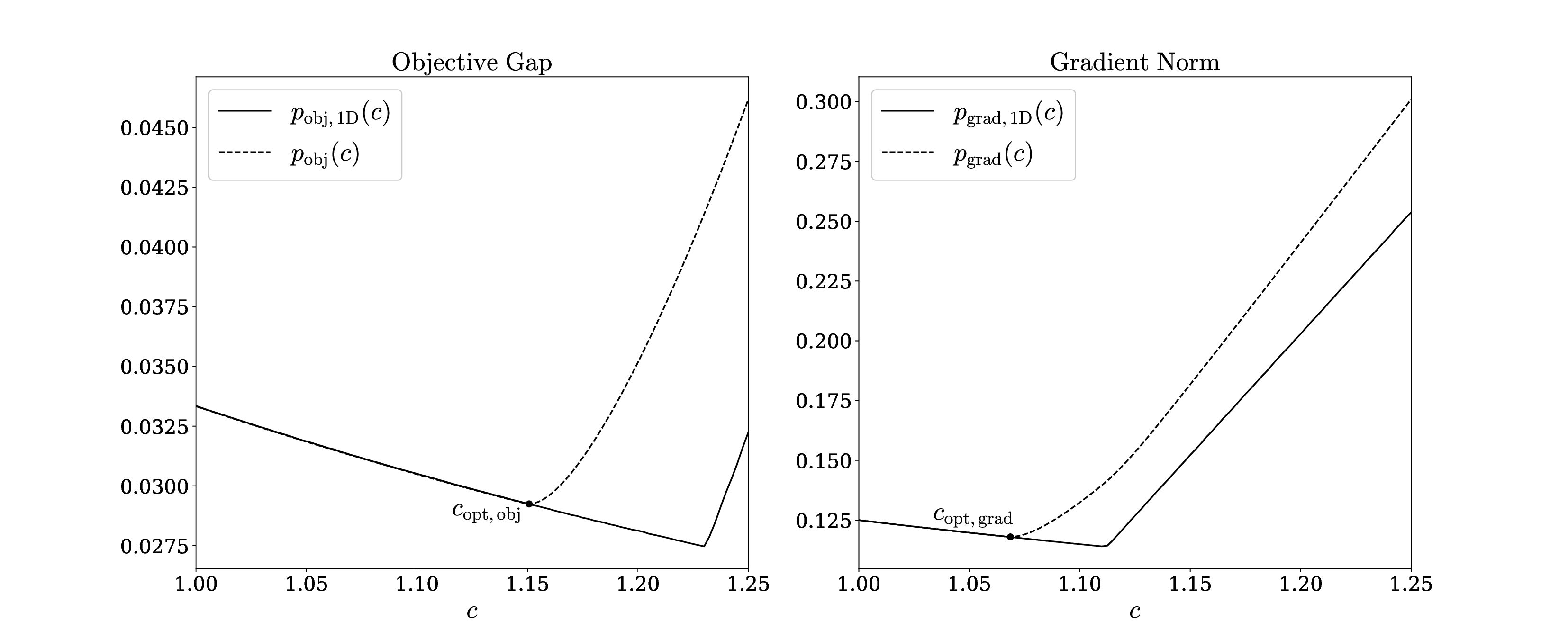}
\caption{Comparison of worst-case performance of smooth functions from the standard PEP \eqref{Eqn:pObj} and from a PEP restricted to one-dimensional functions \eqref{Eqn:primal1D} (and the analogous gradient norm problems) with $N=7$, $L=D=h=1$. For reference, we mark the location of $c_{\text{opt,obj}}$ and $c_{\text{opt,grad}}$. Note the values $c_{\text{crit,obj}}$ and $c_{\text{crit,grad}}$ are not distinguishable from the optimal values in this plot but can be seen in \cref{Fig:cOptZoom}}
\label{Fig:primal1DData}
\end{figure}

\subsection{Application of Simple Extrapolations to Other Algorithms}\label{SubSect:NumericsOther}
Our discussion of extrapolation up to this point has focused on gradient descent with constant stepsizes $h_k=h\in(0,1]$. This simplifying assumption enabled \cref{Sect:SimpleExtrap}'s exact characterization of the impact of extrapolation on worst-case performance. As our final numerical extension, we relax this restriction and consider the impact of simple extrapolation on a range of common first-order methods. In particular, we consider the collection of gradient descent stepsize selections introduced in \cref{SubSec:Tight-Bounds}, Nesterov's classic accelerated method~\cite{Nesterov}, Polyak's heavy ball method~\cite{HeavyBall_Polyak}, and the Optimal Gradient Method~\cite{OGM}. 
The performance estimation SDP~\eqref{Eqn:primal} was previously defined explicitly for gradient descent via the construction of our $\mathbf{x_i}$ vectors, and consequently the matrices $A_{i,j}(h)$ and $B_{i,j}(h)$. This restriction is not fundamental as PEP can apply to any gradient method by adjusting the definition of $\mathbf{x_i}$ to match the method's update.

Although PEP techniques can also naturally generalize to describe proximal or projected gradient methods, extrapolations do not make sense in such settings. Reporting a point outside the convex hull of the algorithm's iterates may set $\xSpecial$ as an infeasible point. Consequently, we focus on unconstrained problems.

The methods considered possess a range of different orders of convergence guarantees for their last iterate. The gradient descent schemes with stepsize bounded above by two and Polyak's heavy ball method~\cite{HeavyBall_Convex} all have the last iterates objective gap converges at a rate of $O(1/N)$ in the worst case. Gradient descent with silver stepsizes~\cite{Silver_Accel} has the last iterate's objective gap converge at rate $O(1/N^{1.2716\dots})$. The last iterate of Nesterov's accelerated method~\cite{Nesterov} attains the optimal order of convergence among all gradient methods of $O(1/N^2)$. Even better, the Optimal Gradient Method (OGM)~\cite{OGM} (discussed using PEP techniques) attains the best possible $O(1/N^2)$ rate exactly among all gradient methods, see~\cite{Drori2016TheEI} for the matching lower bound.

\begin{figure}
\centering
\includegraphics[width=0.8\textwidth]{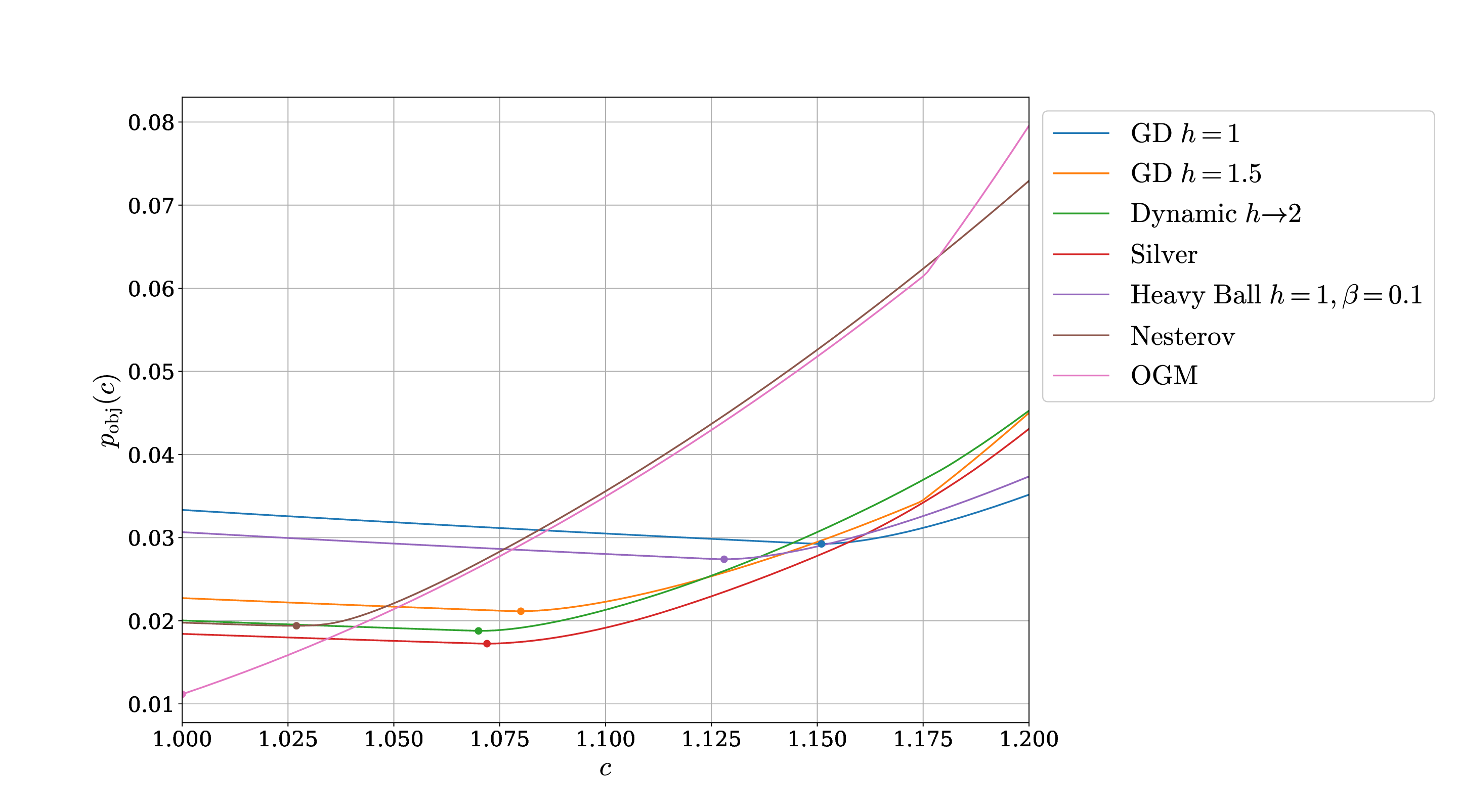}
\caption{Effect of simple extrapolation on worst-case reported objective gap of various gradient methods with $N=7$, $L=D=1$. The optimal choice of $c$ for each algorithm is marked by a dot.}
\label{Fig:AlgComparison}
\end{figure}

\paragraph{Worst-Case Performance} We first consider the effect of simple extrapolation on the worst-case performance of these algorithms, before applying to concrete examples. \cref{Fig:AlgComparison} shows the worst-case objective gap for each considered algorithm as the simple extrapolation factor varies, with $N=7$ and $L=D=1$. We can see that every method except OGM has a non-negligible convergence improvement from applying a simple extrapolation. This holds even for Nesterov acceleration, which already has an optimal convergence order.
Note that, as one should expect, extrapolation cannot improve the convergence guarantee of OGM as the method is already exactly optimal. While exact comparisons between the optimal amount of extrapolation for each algorithm remain dependent on $N$, we find faster (more optimized) algorithms consistently benefited less from extrapolation: Gradient descent with $h=1$'s $O(1/N)$ convergence rate benefits from larger extrapolations than Silver's $O(1/N^{1.2716\dots})$ rate, which in turn benefits from larger extrapolations than Nesterov's $O(1/N^2)$ rate. To this end, \cref{Tbl:cOptData} shows the numerically computed optimal extrapolation factor for each method up to $N=63$. This shows a widespread (small) benefit from simple extrapolation and hints at the potential broad applicability of this as a simple, effectively-free (in computation and memory costs) postprocessing step for smooth convex optimization.

\begin{table}
    \centering\footnotesize
    \begin{tabular}{|c|c|c|c|c|c|c|c|}
    \hline
         N & GD & GD & Dynamic & Silver & Heavy Ball & Nesterov & OGM\\
         & $h=1$ & $h=1.5$ & $h \to 2$ & & $h=1,\beta=0.1$ & & \\
         \hline
         3  & 1.2447 & 1.1033 & 1.0884 & 1.1029 & 1.1855 & 1.1254 & 1.0 \\
         7  & 1.1508 & 1.0795 & 1.0703 & 1.0718 & 1.1279 & 1.0267 & 1.0 \\
         15 & 1.0991 & 1.0585 & 1.0519 & 1.0454 & 1.0868 & 1.0065 & 1.0 \\
         31 & 1.0667 & 1.0419 & 1.0372 & 1.0284 & 1.0596 & 1.0017 & 1.0 \\
         63 & 1.0454 & 1.0297 & 1.0275 & 1.0181 & 1.0407 & 1.0002 & 1.0 \\
         \hline
    \end{tabular}
    \caption{Numerically computed optimal simple extrapolation factor $\cOpt$ for various methods and $N$.}
    \label{Tbl:cOptData}
\end{table}

\paragraph{Real-World Experiments}
Finally, based on this worst-case analysis, we apply simple extrapolation to two types of regression problems. We first consider the logistic regression problem of minimizing
\begin{equation*}
    f(x) := \sum_{i=1}^m \log(1+\exp(-b_i \langle a_i, x \rangle ))
\end{equation*}
for class labels $b_i \in \{-1,1\}$ and feature vector $a_i \in \R^n$. We perform our experiments on the \textit{duke breast-cancer} (referred to as \textit{duke} going forward), \textit{diabetes}, and \textit{breast-cancer} datasets from LIBSVM \cite{LIBSVM}, along with randomly generated data ($x_0 = 0$, $a_i$ with iid standard normal entries, and $b_i$ sampled iid uniformly from $\{-1,1\}$). For our Lipschitz constant, we can set $L = \frac{\|M^T \! M\|}{4}$ where the rows of $M$ are given by $b_i a_i^T$.
Second, we consider least squares with a Huber regularization term weighted by some $\lambda > 0$:
\begin{equation*}
    f(x) := \|Ax-b\|^2 + \lambda \phi_{1,\eta}(x) \ .
\end{equation*}
We consider data from the \textit{bodyfat}, \textit{cpusmall}, and \textit{triazines} datasets from LIBSVM, along with randomly generated data ($x_0 = 0$, $a_i$ and $b_i$ with iid standard normal entries). Viewing this as a smooth analog of Lasso regularization, we follow the convention of \cite{Moreau_LassoConvention}, setting $\lambda = \frac{0.01}{m}\|A^T b\|_\infty$, $\eta = \frac{\lambda}{\|A^T \! A\|}$, and our Lipschitz constant then as $L = 2\|A^T \! A \|$. 

In \cref{Tbl:Experiments}, we report for each of these experiments the relative improvement from extrapolation in comparison to the terminal iterate: $(f(x_N) - f(x_\sigma))/(f(x_N) - f(x_\star))$. We include the relative improvement for the worst-case problem instance, as calculated from PEP, but we note that this is just for reference: since these experiments are not theoretical worst-case instances, one cannot directly compare with the worst-case values. We see a wide range of behavior across our problem types and datasets, emphasizing that our results in \cref{Sect:SimpleExtrap} only guarantee improvements on worst-case behaviors. For logistic regression, the \textit{duke} dataset benefits significantly from simple extrapolation, both at small scale $N=7$ and large scale $N=1000$ (simple extrapolation nearly achieves a 5\% improvement for gradient descent with $h=1$). On the other hand, for the \textit{breast-cancer} dataset, simple extrapolation has a small negative effect. For the least squares problem, extrapolation does not appear to be effective on the real-world data, but performs well on random data. Thus, our numerics show simple extrapolation can be beneficial beyond worst-case analysis, however, its efficacy varies.

\begin{table}
    \centering\scriptsize
    \begin{tabular}{|c|c|c|c|c|c|c|c|c|c|c|}
        \hline
         & & & \multicolumn{4}{|c|}{Logistic Regression} & \multicolumn{4}{|c|}{Least Squares with Huber Regularization} \\
        \hline
             & Algorithm & Worst Case & \textit{duke} & \textit{diabetes} & \textit{breast-cancer}  & Random & \textit{bodyfat} & \textit{cpusmall} & \textit{triazines} & Random\\
        \hline
         & GD $h=1$ & 0.1225 & 0.0908  & 0.0200 & -0.0022  & 0.2671 &
         -1.4275 & -2.0673 & -0.2780 & 0.3500 \\
         & GD $h=1.5$ & 0.0696 & 0.0488 & 0.0127 & -0.0006 & 0.1530 &
         -0.3492 & -0.5070 & -0.0500 & 0.2210 \\
         $N=7$ & Dynamic & 0.0597 & 0.0454 & 0.0120 & -0.0005 &  0.1411 &
         -0.2955 & -0.4294 & -0.0395 & 0.2128 \\
         & Silver & 0.0639 & 0.0442 & 0.0117 & -0.0005  & 0.1392 &
         -0.2771 & -0.4027 & -0.0360 & 0.2086 \\
         & Heavy Ball & 0.1063 & 0.0821 &  0.0178 & -0.0015  & 0.2428 &
         -1.1364 & -1.6468 & -0.2242 & 0.3104 \\
         & Nesterov & 0.0190 & 0.0237 & 0.0051 & -0.0001 & 0.0682  &
         -0.0511 & -0.0743 & -0.0038 & 0.0517 \\
         \hline
         $N=1000$ & GD1 & 0.0099 & 0.0479 & 0.0034 & 0 &  0.0680 & -0.0344 & -0.0105 & -0.0038 & -0.7092 \\
         \hline

    \end{tabular}
    \caption{Experimental results for simple extrapolation across different algorithms. For $N=7$, extrapolation factor is chosen as the numerically optimal $\cOpt$ for each algorithm. For $N=1000$, we use $\cCrit$, as calculated from \eqref{Eqn:psi}. Values show relative change in performance: $(f(x_N) - f(x_\sigma))/(f(x_N) - f(x_\star))$. Results for random data are obtained by solving 100 randomly generated problem instances ($m=50$, $n=200$) and taking the median value of the relative change in performance.}
    \label{Tbl:Experiments}
\end{table}

    \paragraph{Acknowledgements.} This work was supported in part by the Air Force Office of Scientific Research under award number FA9550-23-1-0531. We thank Adrien Taylor for providing kind and direct feedback on this work's positioning among the growing PEP literature.
    
    {\small
    \bibliographystyle{unsrt}
    \bibliography{references}

\begin{thebibliography}{10}

\bibitem{NonsmoothLastIterate}
Moslem Zamani and François Glineur.
\newblock Exact convergence rate of the last iterate in subgradient methods.
\newblock {\em arXiv:2307.11134}, 2023.

\bibitem{grimmer2024primaldual}
Benjamin Grimmer and Danlin Li.
\newblock Some primal-dual theory for subgradient methods for strongly convex optimization.
\newblock {\em arXiv:2305.17323}, 2024.

\bibitem{Gustavsson2015}
E.~Gustavsson, M.~Patriksson, and AB~Str\"omberg.
\newblock Primal convergence from dual subgradient methods for convex optimization.
\newblock {\em Mathematical Programming}, 150:365--390, 2015.

\bibitem{NonsmoothOpt_Lan}
Guanghui Lan.
\newblock {\em First-order and Stochastic Optimization Methods for Machine Learning}.
\newblock Springer, 2020.

\bibitem{Shamir2013}
Ohad Shamir and Tong Zhang.
\newblock Stochastic gradient descent for non-smooth optimization: Convergence results and optimal averaging schemes.
\newblock In {\em Proceedings of the 30th International Conference on Machine Learning}, volume~28, pages 71--79. PMLR, 2013.

\bibitem{AndersonAccel}
Donald~G. Anderson.
\newblock Iterative procedures for nonlinear integral equations.
\newblock {\em Journal of the ACM}, 12(4):547–560, 1965.

\bibitem{AndersonAccel_Recent}
Vien~V. Mai and Mikael Johansson.
\newblock Anderson acceleration of proximal gradient methods.
\newblock In {\em Proceedings of the 37th International Conference on Machine Learning}, ICML'20, pages 6620--6629, 2020.

\bibitem{Teboulle_4T_Analytical}
Marc Teboulle and Yakov Vaisbourd.
\newblock An elementary approach to tight worst case complexity analysis of gradient based methods.
\newblock {\em Mathematical Programming}, 201:63--96, 2023.

\bibitem{Silver_Accel}
Jason~M. Altschuler and Pablo~A. Parrilo.
\newblock Acceleration by stepsize hedging: Silver stepsize schedule for smooth convex optimization.
\newblock {\em Mathematical Programming}, 2024.

\bibitem{FirstPEP}
Yoel Drori and Marc Teboulle.
\newblock Performance of first-order methods for smooth convex minimization: A novel approach.
\newblock {\em Mathematical Programming}, 145:451--482, 2014.

\bibitem{Interpolation}
Adrien Taylor, Julien Hendrickx, and François Glineur.
\newblock Smooth strongly convex interpolation and exact worst-case performance of first-order methods.
\newblock {\em Mathematical Programming}, 161:307--345, 2017.

\bibitem{pepit2022}
Baptiste Goujaud, C\'eline Moucer, Fran\c{c}ois Glineur, Julien Hendrickx, Adrien Taylor, and Aymeric Dieuleveut.
\newblock {PEPit}: computer-assisted worst-case analyses of first-order optimization methods in {P}ython.
\newblock {\em Mathematical Programming}, 16:337--367, 2024.

\bibitem{Interpolation2}
Adrien~B. Taylor, Julien~M. Hendrickx, and Fran\c{c}ois Glineur.
\newblock Exact worst-case performance of first-order methods for composite convex optimization.
\newblock {\em SIAM Journal on Optimization}, 27(3):1283--1313, 2017.

\bibitem{BnB}
Shuvomoy Das~Gupta, Bart~P.G. {Van Parys}, and Ernest~K. Ryu.
\newblock Branch-and-bound performance estimation programming: A unified methodology for constructing optimal optimization method.
\newblock {\em Mathematical Programming}, 204:567--639, 2024.

\bibitem{Grimmer_LongStep}
Benjamin Grimmer.
\newblock Provably faster gradient descent via long steps.
\newblock {\em SIAM Journal on Optimization}, 34(3):2588--2608, 2024.

\bibitem{Grimmer_Accel}
Benjamin Grimmer, Kevin Shu, and Alex~L. Wang.
\newblock Accelerated gradient descent via long steps.
\newblock {\em arXiv:2309.09961}, 2023.

\bibitem{altschuler2018greed}
Jason Altschuler.
\newblock {\em Greed, hedging, and acceleration in convex optimization}.
\newblock PhD thesis, Massachusetts Institute of Technology, 2018.

\bibitem{Kim_ProofOfConstantStepGD}
Jungbin Kim.
\newblock A proof of exact convergence rate of gradient descent. part {II}. performance criterion $(f(x_n)-f_*)/\|x_0-x_*\|^2$.
\newblock {\em arXiv:2412.04427}, 2024.

\bibitem{Wang_ProofOfConjecture}
Bofan Wang, Shiqian Ma, Junfeng Yang, and Danqing Zhou.
\newblock Relaxed proximal point algorithm: Tight complexity bounds and acceleration without momentum.
\newblock {\em arXiv:2410.08890}, 2024.

\bibitem{OGM}
Donghwan Kim and Jeffrey~A. Fessler.
\newblock Optimized first-order methods for smooth convex minimization.
\newblock {\em Mathematical Programming}, 159:81–107, 2016.

\bibitem{Nesterov}
Yurii Nesterov.
\newblock A method for solving the convex programming problem with convergence rate ${O}(1/k^2)$.
\newblock {\em Proceedings of the USSR Academy of Sciences}, 269:543--547, 1983.

\bibitem{Mathematica}
Wolfram~Research{,} Inc.
\newblock Mathematica, {V}ersion 14.0.
\newblock Champaign, IL, 2024.

\bibitem{Drori2019}
Yoel Drori and Adrien~B. Taylor.
\newblock Efficient first-order methods for convex minimization: a constructive approach.
\newblock {\em Mathematical Programming}, 184:183–220, 2020.

\bibitem{LiederThesis}
Felix Lieder.
\newblock {\em Projection Based Methods for Conic Linear Programming}.
\newblock PhD thesis, Heinrich-Heine-Universit\"at D\"usseldorf, 2018.

\bibitem{Drori2014}
Yoel Drori.
\newblock {\em Contributions to the complexity analysis of optimization algorithms}.
\newblock PhD thesis, Tel-Aviv University, 2014.

\bibitem{Barre2023}
Mathieu Barré, Adrien Taylor, and Francis Bach.
\newblock Principled analyses and design of first-order methods with inexact proximal operators.
\newblock {\em Mathematical Programming}, 201:185--230, 2023.

\bibitem{Gu2020}
Guoyong Gu and Junfeng Yang.
\newblock Tight sublinear convergence rate of the proximal point algorithm for maximal monotone inclusion problems.
\newblock {\em SIAM Journal on Optimization}, 30(3):1905--1921, 2020.

\bibitem{Kim2021}
Donghwan Kim.
\newblock Accelerated proximal point method for maximally monotone operators.
\newblock {\em Mathematical Programming}, 190(1–2):57–87, 2021.

\bibitem{mosek}
MOSEK ApS.
\newblock {\em {MOSEK} Optimizer {API} for {J}ulia manual. Version 10.0.}, 2024.

\bibitem{JuMP}
Iain Dunning, Joey Huchette, and Miles Lubin.
\newblock Ju{MP}: A modeling language for mathematical optimization.
\newblock {\em SIAM Review}, 59(2):295--320, 2017.

\bibitem{Hdual}
Jaeyeon Kim, Asuman Ozdaglar, Chanwoo Park, and Ernest~K. Ryu.
\newblock Time-reversed dissipation induces duality between minimizing gradient norm and function value.
\newblock {\em arXiv:2305.06628}, 2023.

\bibitem{HeavyBall_Polyak}
B.T. Polyak.
\newblock Some methods of speeding up the convergence of iteration methods.
\newblock {\em USSR Computational Mathematics and Mathematical Physics}, 4(5):1--17, 1964.

\bibitem{HeavyBall_Convex}
Euhanna Ghadimi, Hamid~Reza Feyzmahdavian, and Mikael Johansson.
\newblock Global convergence of the heavy-ball method for convex optimization.
\newblock In {\em 2015 European Control Conference (ECC)}, pages 310--315, 2015.

\bibitem{Drori2016TheEI}
Yoel Drori.
\newblock The exact information-based complexity of smooth convex minimization.
\newblock {\em Journal of Complexity}, 39:1--16, 2017.

\bibitem{LIBSVM}
Chih-Chung Chang and Chih-Jen Lin.
\newblock {LIBSVM}: A library for support vector machines.
\newblock {\em ACM Transactions on Intelligent Systems and Technology}, 2:27:1--27:27, 2011.

\bibitem{Moreau_LassoConvention}
Thomas Moreau et~al.
\newblock Benchopt: reproducible, efficient and collaborative optimization benchmarks.
\newblock In {\em Proceedings of the 36th International Conference on Neural Information Processing Systems}, NIPS '22. Curran Associates Inc., 2024.

\end{thebibliography}
    }
    \appendix
    \phantomsection
\addcontentsline{toc}{chapter}{Appendix}
\renewcommand{\thesubsection}{A.\arabic{subsection}}
\renewcommand{\theequation}{A.\arabic{equation}}
\renewcommand{\thethm}{A.\arabic{thm}}
\setcounter{equation}{0}
\setcounter{section}{0}
\setcounter{thm}{0}

\setcounter{tocdepth}{1}

\subsection{Verifying Construction of Section~\ref{SubSect:ProofOfExtrap}}\label{App:ScIdentity}
First we claim $\sum_{i\neq j} \lambda_{i,j} a_{i,j} - a_{\star,\idxSpecial} = 0$. This is equivalent to 
\begin{equation*}
    \sum_{j=-1}^N \lambda_{j,k} - \sum_{j=-1}^N \lambda_{k,j} = \begin{cases}
        -1 \quad & \text{if } k=-1 \\
        1 \quad & \text{if } k=N \\
        0 \quad & \text{otherwise} \ .
    \end{cases}
\end{equation*}
One can easily verify this as
\begin{align*}
    & \sum_{j=-1}^N \lambda_{j,-1} - \sum_{j=-1}^N \lambda_{-1,j} = 0 - \left(r_1 + (r_2-r_1) + \dots + (r_N-r_{N-1}) + (1-r_N) \right) = -1 \\
    & \sum_{j=-1}^N \lambda_{j,0} - \sum_{j=-1}^N \lambda_{0,j} = r_1 - r_1 = 0 \\
    & \sum_{j=-1}^N \lambda_{j,k} - \sum_{j=-1}^N \lambda_{k,j} = (r_{k+1}-r_k + r_k) - r_{k+1} = 0 \quad \quad \quad \forall k=1,\dots,N-1 \\
    & \sum_{j=-1}^N \lambda_{j,N} - \sum_{j=-1}^N \lambda_{N,j} = (1-r_N + r_N) - 0 = 1 \ .
\end{align*}

Our second claim is that $Z = \frac{1}{2}S_c(r,t)$. By construction, we can rewrite $S_c(r,t)$ as
\begin{equation}\label{Eqn:ScDecomp}
    S_c(r,t)_{i,j} = \begin{cases}
        t \quad & \text{if } i=j=-1 \\
        -r_1 \quad & \text{if } i=-1,j=0 \text{ or } i=0,j=-1\\
        r_j - r_{j+1} \quad & \text{if } i=-1, 1\leq j \leq N-1 \text{ or } 1\leq i \leq N-1,j=-1 \\
        r_N - 1 \quad & \text{if } i=-1, j=N \text{ or } i=N,j=-1 \\
        \frac{h}{L}(c-r_N) \quad & \text{if } i=N, 0\leq j \leq N-2 \text{ or } 0\leq i\leq N-2,j=N \\
        \frac{2r_{j+1}}{L} \quad & \text{if } 0\leq i=j \leq N-1 \\
        \frac{1}{L} \quad & \text{if } i=j=N \\
        \frac{1}{L}( hr_{j+1}-r_j) \quad & \text{if } i=j+1\leq N-1 \text{ or } j=i+1\leq N-1 \\
        \frac{1}{L}(hc-r_N) \quad & \text{if } i=N, j=N-1 \text{ or } i=N-1,j=N \\
        \frac{h}{L}(r_{j+1}-r_j) \quad & \text{if } j+2 \leq i \leq N-1  \text{ or } i+2 \leq j \leq N-1\ .
    \end{cases}
\end{equation}
Now using our selected dual variables \eqref{Eqn:dualVarVals}, and expanding our special matrices, we can write
\begin{align*}
    Z & = v B_{0,\star} + r_1\left(A_{\star,0} + \frac{1}{2L} C_{\star,0}\right) + (1-r_N)\left(A_{\star,\idxSpecial}+\frac{1}{2L}C_{\star,\idxSpecial}\right) + \sum_{k=1}^{N-1}(r_{k+1}-r_k)\left(A_{\star,k}+\frac{1}{2L}C_{\star,k}\right) \\
    & \qquad + \sum_{k=1}^{N-1} r_k\left(A_{k-1,k} + \frac{1}{2L}C_{k-1,k}\right) + r_N\left(A_{N-1,\idxSpecial} + \frac{1}{2L}C_{N-1,\idxSpecial}\right) \\
    & = \frac{t}{2} \xz \xz^T +\frac{r_1}{2}\left(\gG{0} -(\xz \g{0}^T + \g{0} \xz^T) \right) \\
    & \qquad  + \frac{1-r_N}{2}\left( \frac{1}{L}\gG{\idxSpecial} - (\xz \g{\idxSpecial}^T + \g{\idxSpecial} \xz^T) + \frac{ch}{L}\sum_{l=0}^{N-1} (\gg{\idxSpecial}{l}) \right)\\
    & \qquad + \sum_{k=0}^{N-1}\frac{r_{k+1}-r_k}{2}\left(\frac{1}{L}\gG{k} -(\xz \g{k}^T + \g{k} \xz^T) + \frac{h}{L}\sum_{l=0}^{k-1} (\gg{k}{l}) \right) \\
    & \qquad + \sum_{k=1}^{N-1} \frac{r_k}{2L} \left((h-1)(\gg{k-1}{k}) + \gG{k-1} + \gG{k}\right) \\
    & \qquad + \frac{r_N}{2L} \left(\gG{N-1} + \gG{\idxSpecial} + (hc-1) (\gg{N-1}{\idxSpecial}) + h(c-1)\sum_{l=0}^{N-2} (\gg{l}{\idxSpecial})\right) \ .
\end{align*}
It is now straightforward to check that $Z_{i,j} = \frac{1}{2}S_c(r,t)_{i,j}$ in each of the cases in \eqref{Eqn:ScDecomp}.

\subsection{Proof of Determinant Formula (Lemma~\ref{Lem:detFormula})} \label{App:ProofOfDetFormula}
    For simplicity, we will follow the same notation used by~\cite[Lemma 3.3]{FirstPEP}, but we will assign more generalized values. However, note many of these variables overlap with other unrelated concepts in this paper, so these definitions should be recognized only in the context of this proof.
    
    First, observe that $W_c$ has a very specific structure. Using this structure, we can write the $k$-th principal submatrix for $M_k$ as 
    \begin{equation*}
        M_k = \begin{pmatrix}
            d_0 & a_1 & a_2 & \dots & a_{k-1} & a_k \\
            a_1 & d_1 & a_2 & & a_{k-1} & a_k \\
            a_2& a_2 & d_2 & & a_{k-1} & a_k \\
            \vdots & & & \ddots & & \vdots \\
            a_{k-1} & a_{k-1} & a_{k-1} & & d_{k-1} & a_k \\
            a_k & a_k & a_k& \dots & a_k & d_k
        \end{pmatrix}
    \end{equation*}
    and we consider its determinant. Drori and Teboulle derived the following recursion equation
    \begin{equation}\label{Eqn:DetRecursion}
        \det M_k = \left(d_k - \frac{2a_k^2}{a_{k-1}}+\frac{a_k^2 d_{k-1}}{a_{k-1}^2}\right) \det M_{k-1} - a_k^2\left(1-\frac{d_{k-1}}{a_{k-1}}\right)^2\det M_{k-2}
    \end{equation}
    with $\det M_0 = d_0$ and $\det M_1 = d_0 d_1 - a_1^2$. We set 
    \begin{align*}
        d_i &= 2r_{i+1} = \frac{2c(i+1)}{2Nc-i} \quad && i=0,\dots,N-1 \\
        d_N &= 1 \\
        a_i &= r_{i+1}-r_i = \frac{(i+1)c}{2Nc-i} - \frac{ic}{2Nc-i+1} \quad && i=1,\dots,N-1\\
        a_N &= c-r_N = c-\frac{Nc}{2Nc-N+1} \ .
    \end{align*}
    Then denoting $\alpha_k = d_k - \frac{2a_k^2}{a_{k-1}}+\frac{a_k^2 d_{k-1}}{a_{k-1}^2}$ and $\beta_k = a_k^2\left(1-\frac{d_{k-1}}{a_{k-1}}\right)^2$, for $k=0,\dots,N-1$, we have
    \begin{equation}
    \alpha_k = 4c\frac{\left((2Nc+1)k - k^2-1 \right)}{(2Nc-k)^2}
    \end{equation}
    and
    \begin{equation}
    \beta_k = \frac{c^2\left(4kNc-2Nc-2k^2 +4k- 1\right)^2}{(2Nc-k)^2(2Nc-k+1)^2} \ .
    \end{equation}
    The recursive equation \eqref{Eqn:DetRecursion} can then be expressed nicely as
    \begin{equation}\label{Eqn:DetRecursion2}
    \det M_k = \alpha_k \det M_{k-1} - \beta_k \det M_{k-2} \ .
    \end{equation}
    We further define for $i=0,\dots, N-1$,
    \begin{align*}
    f_i & = c^{i+1} \\
    g_i & = 2Nc-2i-1 \\
    x_i & = \frac{1}{\detDenom{c}} \\
    y_i & = \frac{\detDenom{c}}{(2Nc-i)^2} \ . 
    \end{align*}
    
    Next, we verify that \eqref{Eqn:detMk} holds for the base cases, $M_0$ and $M_1$:
    \begin{align*}
    \det M_0 &= c\left(1+\frac{2Nc-1}{2Nc+1}\right) \frac{2Nc+1}{(2Nc)^2} \\
     & = \frac{1}{N} = d_0\\
     \det M_1 &= c^2\left(1+ \frac{2Nc-3}{2Nc+1} + \frac{2Nc-3}{6Nc-1}\right) \left( \frac{2Nc+1}{(2Nc)^2} \frac{6Nc-1}{(2Nc-1)^2} \right) \\
    & = \frac{28N^2c^2-20Nc-1}{4N^2(2Nc-1)^2} = d_0d_1 - a_1^2 \ .
    \end{align*}
    
    Now, we claim the recurrence~\eqref{Eqn:DetRecursion2} ensures that for $k=0,\dots,N-1$,
    \begin{equation*}
    \det M_k = f_k\left(1+g_k\sum_{i=0}^{k} x_i\right) \prod_{i=0}^k y_i \ .
    \end{equation*}
    This result follows from the argument in~\cite[Lemma 3.3]{FirstPEP}, using our new values for $f_k$, $g_k$, etc., defined above. This completes our proof for $k<N$.
    
    Finally, we consider $M_N$. Note that the formulas for $d_N$ and $a_N$ differ from their counterparts for $k<N$. So to derive an equation for $M_N$, we simply apply the recursion \eqref{Eqn:DetRecursion2} to $M_{N-1}$ and $M_{N-2}$ and simplify. 
    We compute $\alpha_N$ and $\beta_N$ as
    \begin{align*}
    \alpha_N &= d_N - \frac{2a_N^2}{a_{N-1}}+\frac{a_N^2 d_{N-1}}{a_{N-1}^2}
    = \frac{\zeta(c)}{(2Nc+1)^2} \\
    \beta_N &= a_N^2\left(1-\frac{d_{N-1}}{a_{N-1}}\right)^2 
     = \left(2Nc-2N+1\right)^2\frac{c^2\left(4N^2c-2Nc-2N^2 +4N- 1\right)^2}{(2Nc+1)^2 (2Nc-N+1)^2}
    \end{align*}
    where
    \begin{align*}
    \zeta(c) = 16N^3(N-1)c^4 &-8N^2(5N^2-9N+4)c^3+4N(8N^3-22N^2+20N-5)c^2 \\
    &-2(4N^4-16N^3+21N^2-13N+2)c+1 \ .
    \end{align*}
    Applying our recursion \eqref{Eqn:DetRecursion2}, and after significant simplification (see \texttt{Mathematica proof A.2} for verification), we arrive at our desired formula
    \begin{equation*}
    \det M_N(c) = \det W_c = c^N\left(1 - \psiNum{c}\sum_{i=0}^{N-1} x_i \right) \prod_{i=0}^{N-1} y_i \ .
    \end{equation*}

\subsection{Proof of Proposition~\ref{Prop:cCritBd}} \label{App:ProofOfCBound}
Recall our definition of $\psi_N(c)$:
\begin{equation*} 
    \psi_N(c) = 1 - \sum_{i=0}^{N-1} \frac{\psiNum{c}} {\detDenom{c}} \ .
\end{equation*}
Before going forward, we perform a simple change of variables with $\subVar$, where $c=1+\frac{\subVar}{2N}$. This will significantly simplify our analysis to follow. So removing the $N$ subscript, we define $\psi$ as a function of $s$:
\begin{equation}\label{Eqn:psiSubVar}
    \psi_N(c) = \psi(\subVar) := 1 - \frac{1}{2N}\sum_{i=0}^{N-1} \frac{\subVar(\subVar+1)(2N+\subVar+1)}{2N+4Ni+(2i+1)\subVar-2i^2+1} \ .
\end{equation}
We now focus on approximating $\subVarCrit$, where $\subVarCrit$ is the largest root of $\psi(\subVar)$, or equivalently, $\subVarCrit = 2N(\cCrit - 1)$.

Our approach for approximating $\subVarCrit$ will be to find upper and lower bounding functions for $\psi(\subVar)$. We can then find a closed form expression for the roots of these bounding functions, which will in turn act as bounds for $\subVarCrit$.
Define
\begin{equation}\label{Eqn:psiLowerBd}
    \psiL(\subVar) = 1- \left(\harmonicUpperBd \right)\frac{\subVar(\subVar+1)(2N+1)}{4N^2}
\end{equation}
where $\gamma \approx 0.5772$ is the Euler-Mascheroni constant. We claim below that $\psiL(\subVar)$ is a valid lower bound for $\psi(\subVar)$ and defer the calculations to the following section.

\begin{lemma}\label{Lem:psiL}
    For all $\subVar>0$, $\psiL(\subVar) < \psi(\subVar)$.
\end{lemma}

So $\psiL(\subVar)$ is a lower bound for $\psi(\subVar)$, with a much simpler form. Observe that $\psiL(\subVar)$ is quadratic with respect to $\subVar$, it is concave, and $\psiL(0)=1>0$, so it must have a single positive root, which we will call $\subVar_\ell$. This root $\subVar_\ell$ is a lower bound for $\subVarCrit$.
We can easily find $\subVar_\ell$ by the quadratic equation:
\begin{equation*}
    \subVar_\ell = -\frac{1}{2} + \frac{1}{2}\sqrt{1+\frac{16N^2}{(2N+1)(\harmonicUpperBd)}} \ .
\end{equation*}
Therefore, we can define $c_\ell < \cCrit$ by
\begin{equation*}
    c_\ell = 1 -\frac{1}{4N} + \sqrt{\frac{1}{16N^2}+\frac{1}{(2N+1)(\harmonicUpperBd)}} \ .
\end{equation*}
This bound $c_\ell$ gives us the precise form we present in \cref{Prop:cCritBd}.

We now derive an upper bound $c_u$ through a similar approach. Define
\begin{equation}\label{Eqn:psiUpperBd}
    \psiU(\subVar) = 1-\left(\harmonicLowerBd \right) \frac{\subVar(\subVar+1)(3N+1)}{12N^2}
\end{equation}
with $\gamma$ again the Euler-Mascheroni constant. We again defer the calculations to the following section.

\begin{lemma}\label{Lem:psiU}
    For all $0<\subVar\leq N$, $\psiU(\subVar) > \psi(\subVar)$.
\end{lemma}

We claim that the largest root $\subVar_u$ of $\psi_u(\subVar)$ satisfies $\subVar_u \leq N$. Solving our quadratic equation we have
\begin{equation*}
    \subVar_u = -\frac{1}{2} + \frac{1}{2}\sqrt{1+\frac{48N^2}{(3N+1)(\harmonicLowerBd)}}
\end{equation*}
and one can easily check $\subVar_u \leq N$ for $N > 1$. Consequently, by the same argument as for the lower bound, $\subVar_u > \subVar_\mathrm{crit}$. We then define $c_u > \cCrit$ by
\begin{equation*}
    c_u = 1 -\frac{1}{4N} + \sqrt{\frac{1}{16N^2}+\frac{3}{(3N+1)(\harmonicLowerBd)}} \ . 
\end{equation*}

To summarize, from \cref{Lem:psiL,Lem:psiU}, we know for all $N$ that $c_\ell \leq \cCrit \leq c_u$.
The proposition's first claim follows by applying \cref{Thm:MainExtrapThm} to $c_\ell < \cCrit$. The second claim similarly can be obtained by verifying that $1 + 1/(4\sqrt{N \log(N)}) < c_\ell < \cCrit$

Lastly, we observe these upper and lower bounds suffice to determine the asymptotic behavior of $\cCrit$. We claim that $c_\ell = 1 + \Theta\left(\frac{1}{\sqrt{N \log(N)}}\right)$ and similarly $c_u =  1 + \Theta\left(\frac{1}{\sqrt{N \log(N)}}\right)$, from which we can conclude $\cCrit =  1 + \Theta\left(\frac{1}{\sqrt{N \log(N)}}\right)$.
These claimed asymptotic behaviors of the upper and lower bounds follow from the simple limit calculations
\begin{align*}
    \lim_{N\to\infty} \frac{c_\ell-1}{\frac{1}{\sqrt{N\log(N)}}} & = \lim_{N \to \infty} \left(-\frac{1}{4N} + \sqrt{\frac{1}{16N^2}+\frac{1}{(2N+1)(\harmonicUpperBd)}}\right) \sqrt{N \log(N)} \\
    & = \lim_{N\to\infty} - \frac{\sqrt{N \log(N)}}{4N} + \sqrt{\frac{N\log(N)}{16N^2} + \frac{N \log(N)}{(2N+1)(\harmonicUpperBd)}} = \frac{\sqrt{2}}{2} \ , \\
    \lim_{N\to\infty} \frac{c_u-1}{\frac{1}{\sqrt{N\log(N)}}} & = \lim_{N \to \infty} \left(-\frac{1}{4N} + \sqrt{\frac{1}{16N^2}+\frac{3}{(3N+1)(\harmonicLowerBd)}}\right) \sqrt{N \log(N)} \\
    & = \lim_{N\to\infty} - \frac{\sqrt{N \log(N)}}{4N} + \sqrt{\frac{N\log(N)}{16N^2} + \frac{3N \log(N)}{(3N+1)(\harmonicLowerBd)}} = 1 \ .
\end{align*}

\subsection{Proof of Lemma~\ref{Lem:psiL}}

    We prove this lower bound in two steps, using an intermediate function $\psiLHat(\subVar)$. Define 
    \begin{equation*}
        \psiLHat(\subVar) = 1-\frac{1}{2N}\sum_{i=0}^{N-1}\frac{\subVar(\subVar+1)(2N+1)}{2N+4Ni-2i^2+1} \ .
    \end{equation*}
    We first show that $\psiLHat(s) \leq \psi(s)$.
    We calculate
    \begin{align*}
            \psi(\subVar) - \psiLHat(s) &= \left(1 - \frac{1}{2N}\sum_{i=0}^{N-1} \frac{\subVar(\subVar+1)(2N+\subVar+1)}{2N+4Ni+(2i+1)\subVar-2i^2+1}\right) - \left(1-\frac{1}{2N}\sum_{i=0}^{N-1}\frac{\subVar(\subVar+1)(2N+1)}{2N+4Ni-2i^2+1}\right) \\
            & =\frac{\subVar(\subVar+1)}{2N} \left(\sum_{i=0}^{N-1}\frac{2N+1}{2N+4Ni-2i^2+1} - \sum_{i=0}^{N-1} \frac{2N+\subVar+1}{2N+4Ni+(2i+1)\subVar-2i^2+1} \right) \\
            & = \frac{\subVar(\subVar+1)}{2N} \left(\sum_{i=0}^{N-1} \frac{2i(i+1)\subVar}{(2N+4Ni+(2i+1)\subVar-2i^2+1)(2N+4Ni-2i^2+1)} \right)\\
            & \geq 0 \ .
    \end{align*}
    For the final inequality, we use the fact that both terms of the denominator are always positive for $i=0,\dots,N$.
    Next we show that $\psiL(\subVar) < \psiLHat(\subVar)$. Recalling our definition of $\psi(s)$ from \eqref{Eqn:psiSubVar}, we have:
    \begin{align*}
            \psiLHat(\subVar) & = 1-\frac{1}{2N}\sum_{i=0}^{N-1}\frac{\subVar(\subVar+1)(2N+1)}{2N(1+2i-i\frac{i}{N}+\frac{1}{2N})}\\
            & > 1-\frac{1}{2N}\sum_{i=0}^{N-1}\frac{\subVar(\subVar+1)(2N+1)}{2N(1+i)} \\
            & = 1-\frac{\subVar(\subVar+1)(2N+1)}{4N^2}\sum_{i=0}^{N-1} \frac{1}{i+1} \\
            & \geq 1-\frac{\subVar(\subVar+1)(2N+1)}{4N^2}\left( \harmonicUpperBd \right)\\
            & = \psiL(\subVar)
    \end{align*}
    where the second inequality uses the harmonic series upper bound $\sum_{i=1}^N \frac{1}{i} \leq \harmonicUpperBd$.

\subsection{Proof of Lemma~\ref{Lem:psiU}}
    We again define an intermediate function $\psiUHat(\subVar)$:
    \begin{equation*}
        \psiUHat(\subVar) = 1-\frac{1}{2N}\sum_{i=0}^{N-1} \frac{\subVar(\subVar+1)(3N+1)}{3N+6Ni-2i^2+1} \ .
    \end{equation*}
    We first show that $\psiUHat(\subVar) \geq \psi(\subVar)$. Recalling our definition of $\psi(s)$ from \eqref{Eqn:psiSubVar}, we have:
    \begin{align*}
            \psi(\subVar) - \psiUHat(\subVar) &= \left(1 - \frac{1}{2N}\sum_{i=0}^{N-1} \frac{\subVar(\subVar+1)(2N+\subVar+1)}{2N+4Ni+(2i+1)\subVar-2i^2+1}\right) - \left(1-\frac{1}{2N}\sum_{i=0}^{N-1} \frac{\subVar(\subVar+1)(3N+1)}{3N+6Ni-2i^2+1}\right) \\
            & = \frac{\subVar(\subVar+1)}{2N} \left( \sum_{i=0}^{N-1} \frac{3N+1}{3N+6Ni-2i^2+1} - \frac{2N+\subVar+1}{2N+4Ni+(2i+1)\subVar-2i^2+1}\right) \\
            & = \frac{\subVar(\subVar+1)}{2N} \left( \sum_{i=0}^{N-1} \frac{-2i(i+1)(N-\subVar)}{(3N+6Ni-2i^2+1)(2N+4Ni+(2i+1)\subVar-2i^2+1}\right) \\
            & \leq 0
    \end{align*}
    where we use the fact that $s \leq N$.
    Finally, we show that $\psiU > \psiUHat(\subVar)$.
    \begin{align*}
        \psiUHat(\subVar) & = 1 - \frac{1}{2N}\sum_{i=0}^{N-1}\frac{\subVar(\subVar+1)(3N+1)}{3N(1+2i-\frac{2i^2}{3N}+\frac{1}{3N})} \\
        & < 1 - \frac{1}{2N}\sum_{i=0}^{N-1}\frac{\subVar(\subVar+1)(3N+1)}{3N(2+2i)} \\
        & = 1-\frac{\subVar(\subVar+1)(3N+1)}{12N^2}\sum_{i=0}^{N-1}\frac{1}{i+1} \\
        & \leq 1-\frac{\subVar(\subVar+1)(3N+1)}{12N^2}\left( \harmonicLowerBd \right) \\
        & = \psiU(\subVar) 
    \end{align*}
    where the second inequality uses the harmonic series lower bound $\sum_{i=1}^N \frac{1}{i} \geq \harmonicLowerBd$.

\end{document}